\documentclass[]{siamart171218}

\usepackage{amsfonts,amsmath,amssymb,bm}
\usepackage{color,cases,float}
\usepackage{graphicx}
\usepackage{hyperref}
\usepackage{psfrag}
\usepackage{eucal}
\usepackage{algorithm}
\usepackage{overpic}
\usepackage{subfigure}
\usepackage{url}
\usepackage[all,cmtip]{xy}
\definecolor{darkblue}{rgb}{0.0,0.0,0.6}
\hypersetup{colorlinks,breaklinks,
	linkcolor=darkblue,urlcolor=darkblue,
	anchorcolor=darkblue,citecolor=darkblue}

 \newtheorem{cor}{Corollary}
 \newtheorem{prop}{Proposition}
 \newtheorem{assum}{Assumption}
 
 \newtheorem{lem}{Lemma}
 \newtheorem{rem}{Remark}

\def\0{{\bf 0}}
\def\1{{\bf 1}}

\def\bes{\begin{equation*}}
	\def\ees{\end{equation*}}
\def\be{\begin{equation}}
	\def\ee{\end{equation}}
\def\bas{\begin{eqnarray*}}
	\def\eas{\end{eqnarray*}}
\def\ba{\begin{eqnarray}}
	\def\ea{\end{eqnarray}}
\def\bma{\begin{bmatrix}}
	\def\ema{\end{bmatrix}}
\def\bmx{\begin{matrix}}
	\def\emx{\end{matrix}}
\def\ben{\begin{enumerate}}
	\def\een{\end{enumerate}}
\def\bit{\begin{itemize}}
	\def\eit{\end{itemize}}
\def\bet{\begin{tabular}}
	\def\eet{\end{tabular}}

\def\qed{\hfill \vrule height6pt width 6pt depth 0pt}

\def\e{\epsilon}
\def\de{\delta}

\def\an#1{{\color{black}#1}}
\def\tat#1{{\color{black}#1}}
\def\tatiana#1{{\color{black}#1}}

\def\R{\mathbb{R}}

\def\d{\delta}
\def\la{\langle}
\def\ra{\rangle}
\def\b{{\beta}}
\def\a{\alpha}
\def\g{\gamma}
\def\m{\mu}

\def\dist{{\textrm {dist}}}
\def\as{\emph{a.s.}}

\def\dom{{\textrm {dom}}}
\def\grad{\nabla}

\def\ex#1{\mathop \mathbb{E}\left[ {#1}\right] }

\begin{document}


\ifpdf
  \DeclareGraphicsExtensions{.eps,.pdf,.png,.jpg}
\else
  \DeclareGraphicsExtensions{.eps}
\fi


\newsiamremark{remark}{Remark}
\newsiamremark{hypothesis}{Hypothesis}
\crefname{hypothesis}{Hypothesis}{Hypotheses}
\newsiamthm{claim}{Claim}

\headers{Huber Penalty Approach to Problems with Linear Constraints}{A. Nedi\'c and T. Tatarenko}

\title{Huber Loss-Based Penalty Approach to Problems with Linear Constraints\thanks{Submitted to the editors on ......}}

\author{
 Angelia Nedi\'c\thanks{The School of Electrical, Computer and Energy Engineering at Arizona State University, USA
  (\email{Angelia.Nedich@asu.edu}). A.~Nedi\'c gratefully acknowledges support of this work by the
Office of Naval Research grant no.\ N00014-12-1-0998.}
\and
Tatiana Tatarenko\thanks{Department of Control Theory and Robotics, TU Darmstadt, Germany
	(\email{tatiana.tatarenko@tu-darmstadt.de}). T.~Tatarenko gratefully acknowledges support of this work by the German Research Foundation
	(Deutsche Forschungsgemeinschaft, DFG) within the SPP 1984 “Hybrid and
	multimodal energy systems: System theoretical methods for the transformation
	and operation of complex networks”.}
}



\ifpdf
\hypersetup{
  pdftitle=Huber Loss-Based Penalty Approach to Problems with Linear Constraints,
  pdfauthor={Tatarenko, Nedi\'c}
}
\fi




\maketitle
\begin{abstract}
	We consider an optimization problem with many linear inequalities constraints. 
To deal with a large number of constraints, we provide a
penalty reformulation of the problem, where the penalty is a variant of the one-sided Huber loss function.
We study the infeasibility properties of the solutions of penalized problems for nonconvex and convex objective functions, as the penalty parameters vary with time. 
Then, we propose a random incremental penalty method for solving convex problem,
and investigate its convergence properties for convex and strongly convex objective functions. We show that the iterates of the method converge to a solution of the original problem almost surely and in expectation for suitable choices of the penalty parameters and the stepsize. Also, we establish convergence rate of the method for appropriately defined weighted averages of the iterates for the expected function values. 
We establish  $O(\ln^{1/2+\e} k/{\sqrt k})$-convergence rate when the objective function is convex, and  $O(\ln^{\e} k/k)$-convergence rate when the objective function is strongly convex with $\e$ being any small positive number. 
To the best of our knowledge, these are the first results on the convergence rate for the penalty-based incremental subgradient method with time-varying penalty parameters.
\end{abstract}



\section{Introduction}
In  this  paper,  we  study  the  problem  of minimizing  a
\emph{convex} function $f:\R^n\to\R$ over a convex and closed set $X$
that is the intersection of finitely many {sets} $X_i$, represented by linear inequalities, $i=1,\ldots,m,$
where $m\ge 2$ is large,
i.e., 
\ba\label{eq:gprob}
\hbox{min} \, f(x), \quad \hbox{s.t.} \,  x\in X = \cap_{i=1}^m X_i. \vspace{-0.3cm}
\ea
Optimization problems of the form~\eqref{eq:gprob} arise in many areas of research,
such as digital filter settings in communication systems \cite{filter}, energy consumption in Smart Grids \cite{SmartG}, convex  relaxations of  various  combinatorial  optimization  problems in machine learning applications \cite{clustering, matching}.

Our interest is in case when $m$ is large, which prohibits us from using
projected gradient and
augmented Lagrangian methods~\cite{BertsekasConstrOpt, Xu2020PrimalDualSG},
that require either computation of the (Euclidean) projection or
an estimation of the gradient for the sum of many functions, at each iteration.
To reduce the complexity, one may consider a method that operates on a  single set $X_i$
from the constraint set collection $\{X_1,\ldots,X_m\}$ at each iteration.
Algorithms using random constraint sampling for general convex optimization problems~\eqref{eq:gprob}  have been first considered  in~\cite{Nedich2011} and were extended in~\cite{WangBertsSM} to a broader class of randomization over the sets of constraints. The convergence rate analysis in~\cite{Nedich2011,WangBertsSM}
	demonstrates that the expected optimality error (in terms of function values) diminishes to zero with the rate of $O(1/\sqrt k)$
	which is the optimal convergence rate for merely convex functions. 
The work~\cite{Patrascu2017NonAsConv} further extends the results from~\cite{Nedich2011,WangBertsSM} to non asymptotic analysis of problems with potentially infinitely many constraints. On the other hand, a primal-dual method updating a random coordinate of the dual variable pro iteration was proposed for generalized linear programming with equality constraints in \cite{Song2021}. 

A possible reformulation of the problem~\eqref{eq:gprob} is through the use of the indicator functions of
the constraint sets, resulting in the following unconstrained problem
\be\label{eq:reform}
\min_{x\in\R^n}\sum_{i=1}^m \left\{\frac{1}{m}f(x) + \chi_i(x)\right\},
\ee
where $\chi_i(\cdot):\R^n\to\R\cup\{+\infty\}$ is the indicator function of the set $X_i$
(taking value $0$ at the points
$x\in X_i$ and, otherwise, taking value $+\infty$).
The advantage of this reformulation is that the objective function is the sum of convex functions and incremental methods can be employed that compute only a (sub)-gradient of one of the component functions at each iteration.
The traditional incremental methods do not have memory, and their origin can be traced back to work of Kibardin~\cite{Kibardin}. They have been studied for
smooth least-square problems~\cite{Ber97,Luo91}, for training the neural networks~\cite{Gri94,Gri00}, for smooth convex problems~\cite{Sol98,Tse98} and
for non-smooth convex problems~\cite{NeB01,GGM06,HeD09,JRJ09,Wright08}
(see~\cite{BertsekasPenalty} for a more comprehensive survey of these methods).
However, no rate of convergence to the exact solution has been obtained for such procedures.
Reformulation~\eqref{eq:reform} has been considered in~\cite{Kundu2018} as a departure point
toward an exact penalty reformulation using the set-distance functions.
This exact penalty formulation has been motivated by a simple exact penalty model proposed in~\cite{Bertsekas2011} (using only the set-distance functions)
and a more general penalty model considered in~\cite{BertsekasPenalty}.
In~\cite{Kundu2018}, a lower bound on the penalty parameter
has been identified guaranteeing that the optimal solutions of the penalized problem are also optimal solutions of the original problem~\eqref{eq:reform}. However, this bound depends on a so-called  regularity constant for the constraint set, which might be difficult to estimate. Moreover, the proposed approaches in~\cite{Kundu2018} do not utilize
incremental processing, but rather primal-dual approaches where a full (sub)-gradient of the penalized function
is used.
On a broader scale, our work is related to random methods for solving linear feasibility problems~\cite{Nedic-cdc-2010,Strohmer}, and their extensions to solving convex  inequality systems~\cite{NecRichPat2019}.

In contrast to the penalized formulation in~\eqref{eq:reform} and the
works mentioned above, this paper deals with a penalized reformulation of the problem~\eqref{eq:gprob}, where the penalized problems vary with time.
This is done by varying penalty parameters so as to gradually decrease the infeasibility of the iterates. In this way,  we can guarantee convergence of the  \emph{single time scale
procedure incremental procedure} to an exact solution of the original problem~\eqref{eq:gprob}. Our choice of the penalty functions is a variant of the one-sided Huber losses~\cite{Huber}, which have Lipschitz continuous gradients.
In the work~\cite{Penalty_siam}, existence of the fixed penalty choices for this penalized reformulation has been shown under which the fast incremental algorithms can be applied to achieve convergence to a feasible point in a predefined neighborhood of the optimal solution of the original problem, with a linear convergence rate. However, to guarantee this convergence, some problem specific parameters need to be known, which are difficult to estimate in practice. The recent paper~\cite{PenaltyConvRate-cdc} deals with the penalty parameters which vary with time. However, that work considers a strongly convex objective function $f(\cdot)$ exclusively. 
In this present paper, we \an{show that some of the results in~\cite{PenaltyConvRate-cdc} extend to nonconvex objective functions $f(\cdot)$ and to merely convex functions, as given in Section~\ref{subsec-var-par-a} and
	Section~\ref{subsec-var-par-b}, respectively.} 
The random incremental penalty-based gradient method has been proposed in~\cite{PenaltyConvRate-cdc} and analyzed for a strongly convex objective function $f(\cdot)$ with Lipschitz continuous gradients. In this paper, however,
we consider a subgradient variant of the method in Section~\ref{sec-algo} for both convex and strongly convex $f(\cdot)$. \an{Therein, we analyze its almost sure  convergence and show its convergence rate of the order $O(\ln^{1/2+\e} k/\sqrt{k})$ for merely convex objective function, where  $\epsilon>0$ is arbitrarily small. Moreover, for strongly convex objective function, we establish convergence rate 
	of the order $O(\ln^{\e} k /{k})$ for arbitrarily small $\epsilon>0$, which improves the convergence rate of $O(1/\sqrt{k})$ provided in~\cite{PenaltyConvRate-cdc}.} Note that these rates possess better dependence on the logarithmic term than the rates obtained in \cite{Fercoq2019AlmostSC}. In that work the authors focus on an optimization problem over a convex set with a special structure and present a  penalty-based method with a double loop structure, where at each iteration the corresponding penalized problem has
to be solved up to some accuracy, and prove the convergence rates of the order $O(\ln k /{\sqrt k})$ and $O(\ln k /{k})$ for the purely convex and strongly convex cases respectively.


The outline of the paper is as follows. In Section~\ref{sec-formulation},
we provide the penalty based formulation of the original problem~\eqref{eq:gprob}
and some basic properties of the chosen penalty functions.
In Section~\ref{sec-var-par}, we investigate the relations for the solutions of the penalized problems, as we vary the penalty parameters, for continuous nonconvex objective function $f(\cdot)$ and for convex objective function $f(\cdot)$.
In Section~\ref{sec-algo}, we propose a random incremental penalty method and show that its iterates converge almost surely and in expectation to a solution of the original problem, under suitable assumptions on the penalty parameters and the stepsize. In Section 4, we also provide convergence rate estimates using appropriate weighted averages of the iterates. To the best of our knowledge, these are the first results on the convergence rate for the penalty-based incremental subgradient method with time-varying penalty parameters.
In Section~\ref{sec-concl}, we conclude the paper.


\section{Problem Formulation and its Penalty-based Reformulation}\label{sec-formulation}
We consider the following optimization problem:
\ba\label{eq:problem}
\hbox{min} \,  f(x), \quad \hbox{s.t.} \,  \la a_i,x\ra - b_i\le 0, \  i\in[m],
\ea
where $[m]=\{1,\ldots,m\}$ and 
the vectors $a_i\in\R^n$, $i\in[m]$, are nonzero.
We will assume that the problem is \emph{feasible}.
Throughout the paper, we use $X_i$ to denote the set of points satisfying the $i$-th inequality constraint, i.e.,
$X_i=\{x\in\R^n\mid  \la a_i,x\ra - b_i\le 0\} \quad\hbox{for all  }i\in[m],$
and $X$ to denote the nonempty intersection of the sets $X_i$, $i\in[m]$, i.e., $X=\cap_{i\in[m]}X_i$.
Associated with problem~\eqref{eq:problem}, we
consider a penalized problem\vspace{-0.2cm}
\ba\label{eq:pen-problem}
\min_{x\in\R^n}F_{\g\d}(x),
\ea
where\vspace{-0.2cm}
\begin{equation}\label{eq:penfun}
	F_{\g\d}(x) = f(x) + \frac{\gamma}{m} \sum_{i=1}^{m} h_\d\left(x; a_i, b_i\right).\end{equation}
Here, $\gamma>0$ and $\d\ge0$ are penalty parameters.
The vectors $a_i$ and scalars $b_i$ are the same as those characterizing the constraints in problem~\eqref{eq:problem}.
For a given nonzero vector $a\in\R^n$ and $b\in\R$, the penalty function
$h_\d(\cdot;a,b)$ is given by
\vspace{-0.2cm}
\begin{align}\label{eq:hfun}
	h_\delta(x; a,b) = \begin{cases}
		\frac{\la a,x\ra -b}{\|a\|} &\text{ if }  \la a,x\ra - b>\de,\\
		\frac{(\la a,x\ra -b + \de)^2}{4\de\|a\|} &\text{ if }  -\de\le\la a,x\ra - b\le\de,\\
		0 &\text{ if }  \la a,x\ra - b<-\de,
	\end{cases}
\end{align}
(see Figure~\ref{fig:penalty} for an illustration).
\begin{figure}[!t]
	\centering
	\psfrag{-0.5}[c][l]{\scriptsize{$-0.5$}}
	\psfrag{0}[c][l]{\scriptsize{$0$}}
	\psfrag{0.5}[c][l]{\scriptsize{$0.5$}}
	\psfrag{1}[c][l]{\scriptsize{$1$}}
	\psfrag{1.5}[c][l]{\scriptsize{$1.5$}}
	\psfrag{2}[c][l]{\scriptsize{$2$}}
	\psfrag{h}[c][l]{{\Large{$h_{\delta}$}}}
	\psfrag{x}[c][b]{\Large{$x$}}
	\begin{overpic}[width=0.81\linewidth]{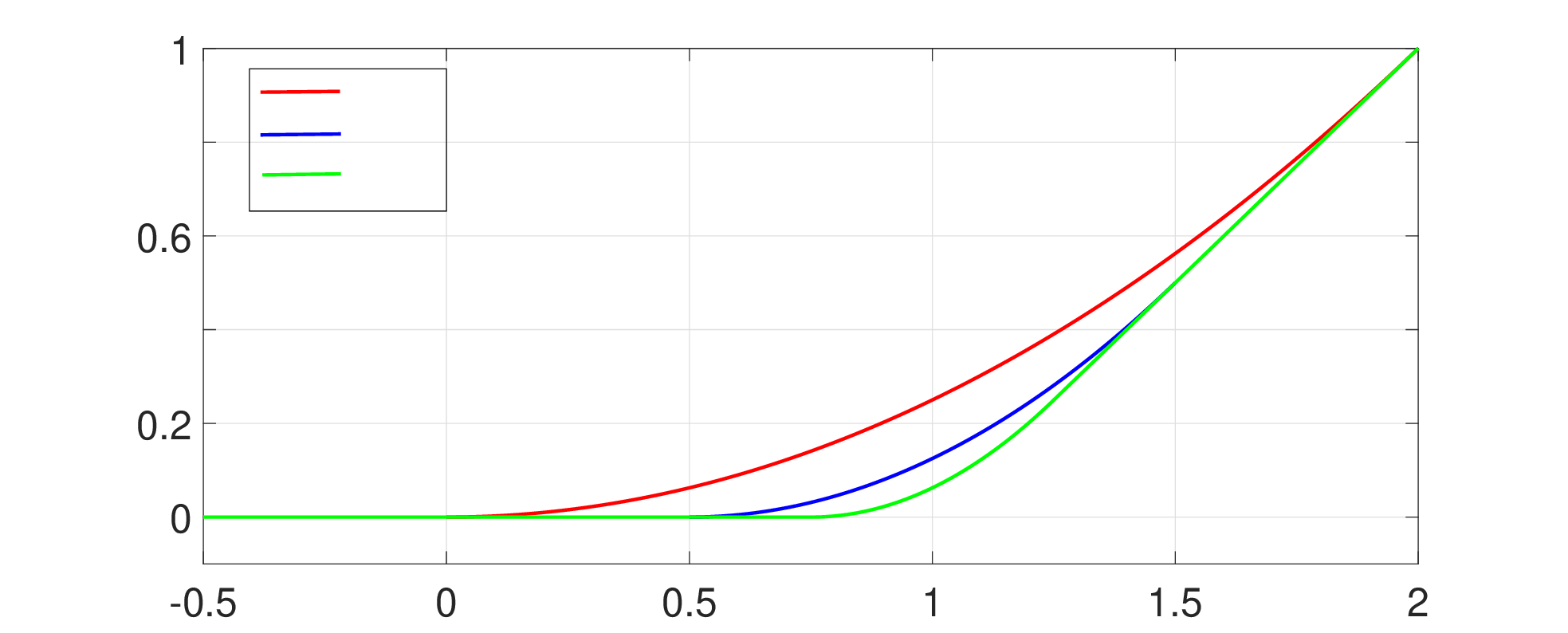}
		\put(22,34){\scriptsize{$\delta =1$}}
		\put(22,31.4){\scriptsize{$\delta =\frac12$}}
		\put(22,28.5){\scriptsize{$\delta =\frac14$}}
		\put(51,-1){\Large{$x$}}
		\put(4,17){\rotatebox{90}{\Large$h_{\d}$}}
	\end{overpic}
	\caption{Penalty functions $h_\delta(x;1,1)$ for the constraint $x-1\le 0$, $x\in\R$,
		with $\d\in\left\{\frac{1}{4},\frac{1}{2},1\right\}$.}
	\label{fig:penalty}
\end{figure}
The penalty function $h_\d(\cdot; a,b)$ is a variant of the one-sided Huber loss functions.
Originally, such functions have been introduced in applications of robust regression models
to make them less sensitive to outliers in data in comparison with the squared error loss~\cite{Huber}.
In contrast, here,
we use this type of penalty function to smoothen the distance-based penalties (the average distance to the sets $X_i$)
proposed in~\cite{BertsekasPenalty}.
Furthermore, an appropriate choice of the parameter $\delta\ge0$ allows us to increase the penalty values as compared to the penalties based on the squared distances to the sets $X_i$, which have a small penalty values around an optimum lying close to the boundary of the constraint set~\cite{Siedlecki}.

For any $\d\ge0$, the function $h_\d(\cdot;a,b)$ satisfies the following relations:
\begin{align}\label{eq:hfunineq}
	h_\delta(x; a,b) \ge 0\qquad\hbox{for all }x\in\R^n,
\end{align}\vspace{-0.2cm}
\begin{align}\label{eq:hfunineq1}
	h_\delta(x; a,b) \le \frac{\delta}{4\|a\|}\qquad\hbox{for all $x$ satisfying }
	\la a,x\ra\le b,
\end{align}\vspace{-0.2cm}
\begin{align}\label{eq:hfunineq2}
	h_\delta(x; a,b) > \frac{\delta}{4\|a\|}\qquad\hbox{for all $x$ satisfying }
	\la a,x\ra> b.
\end{align}

For a vector $a\ne0$, the function $h_\delta(\cdot; a,b)$ is a composition of
a scalar function\vspace{-0.2cm}
\begin{align}\label{eq:sfun}
	p_\d(s)= \begin{cases}
		s &\text{ if } \quad s>\d,\\
		\frac{(s + \de)^2}{4\d} &\text{ if } \quad -\de\le s \le\de,\\
		0 &\text{ if } \quad s<-\de,
	\end{cases}
\end{align}
with a linear function  $x\mapsto \la a,x\ra-b$, scaled by $\frac{1}{\|a\|}$, i.e.,
\be\label{eq:handp}
h_\delta(x; a,b) =\frac{1}{\|a\|}p_\d(\la a,x\ra-b)\qquad\hbox{for all }x\in\R^n.\ee
The function $p_\d(\cdot)$ is convex on $\R$ for any $\delta\ge0$,
implying that the function $h_\d(\cdot;a,b)$ is convex on $\R^n$ for any $\d\ge0$.

Furthermore, the function
$p_\d(\cdot)$ is twice differentiable for any $\d>0$,
with the first and second derivatives given by
\begin{align}\label{eq:pderiv}
	p'_\d(s)= \begin{cases}
		1 &\text{ if } \quad s>\d,\\
		\frac{(s + \de)}{2\d} &\text{ if } \quad -\de\le s \le\de,\\
		0 &\text{ if } \quad s<-\de,
	\end{cases}
\end{align}
\begin{align*}
	p''_\d(s) = \begin{cases}
		\frac{1}{2\d} &\text{ if } \quad -\d\le s\le\d,\\
		0 &\text{ if } \quad s<-\de \quad \text{or}\quad s>\d.
	\end{cases}
\end{align*}
Thus, for $\d>0$, the function $p_\d(\cdot)$ has Lipschitz continuous derivatives with a constant $\frac{1}{2\d}$.
Hence, the function $h_\delta(\cdot; a,b)$ is differentiable for any $\d>0$, and
its gradient is given by
\be\label{eq:gradh}
\nabla h_\delta(x; a,b) =\frac{1}{\|a\|}\,p'_\d(\la a,x\ra-b) a\qquad\hbox{for all }x\in\R^n.\ee
Moreover,  $\nabla h_\delta(\cdot; a,b)$ is Lipschitz continuous with a constant $\frac{\|a\|}{2\d}$, i.e.,
for all $x,y\in\R^n$,
\be\label{eq:Lipc-gradh}
\|\nabla h_\delta(x; a,b) -\nabla h_\delta(y; a,b)\| \le \frac{\|a\|}{2\d}\,\|x-y\|.
\ee
In view of the definition of the penalty function $F_{\g\d}$ in~\eqref{eq:penfun}
and relation~\eqref{eq:gradh},
we can see that the magnitude of the ``slope" of the penalty function is controlled by the parameter $\g>0$,
while the ratio of the parameters $\g$ and $\d$ is controlling the ``curvature" of the penalty function.
Finally, in the following lemma, we provide some additional properties of the gradients $\nabla h_{\d}(\cdot;a,b)$.
\begin{lem}\label{lem:pderiv}
	Consider the function $h_\d(\cdot; a,b)$ as given in~\eqref{eq:hfun}.
	Then, we have $\|\nabla h_\d(x;a,b)\|\le 1$ for all $x\in\R^n$. Additionally,
	if $\d_1 \ge \d_2$, then
	\[\max_{x\in\R^n}\|\nabla h_{\d_1}(x;a,b)- \nabla h_{\d_2}(x;a,b)\|
	\le \frac{\d_1-\d_2}{2\d_1}.\]
\end{lem}
\begin{proof} Can be found in \cite{Penalty_siam}. \end{proof} 

Throughout the rest of the article, 
we let $\Pi_Y [x]$ denote the (Euclidean) projection of a point $x$ on a convex closed set $Y$,
i.e.,
$\dist(x, Y) = \|x - \Pi_{Y} [x]\|.$ 
Also, the smallest norm of the vectors $a_i$, $i\in[m]$, is denoted by $\a_{\min}$, i.e., 
$\a_{\min}=\min_{i\in[m]}\|a_i\|$.

\section{Time-varying Penalty Parameters}\label{sec-var-par}
Consider positive scalar sequences $\{\delta_k\}$ and $\{\g_k\}$
and, for each $k$, let  $F_k$ denote the penalty function $F_{\d_k\g_k}(x)$, i.e.,
\begin{equation}\label{eq-penalized-function}
	F_k(x) = f(x) + \frac{\g_k}{m} \sum_{i=1}^{m} h_k\left(x; a_i, b_i\right),
\end{equation}
where we use $h_k\left(\cdot; a_i, b_i\right)$ to denote $h_{\d_k}\left(\cdot; a_i, b_i\right)$ for each~$i$.
We let $X_k^*$ denote the optimal set for the penalized problem $\min_{x\in\R^n} F_k(x)$ and $X^*$  denote the optimal set for the original problem~\eqref{eq:problem}. We next investigate the properties of the optimal solutions of the penalized problems for the case of a general function $f(\cdot)$  and the case when $f(\cdot)$ is convex.
\subsection{General Function $f$}\label{subsec-var-par-a}
When $f(\cdot)$ is continuous and has bounded lower-level sets, 
the optimal set  $X^*$ for the original problem is nonempty and compact. Also, for each $k$,
the optimal set $X_k^*$ of the penalized problem is also nonempty and compact.

As a solution $x_k^*\in X_k^*$ need not be feasible, we next estimate the distance of any $x_k^*\in X_k^*$
from the feasible set $X$. 
To simplify the notation, we denote by $H_k(\cdot)$ the average of the penalty functions $h_k(\cdot;a_i,b_i)$, $i\in[m]$,
\begin{equation}\label{eq-hk}
	H_k(x)= \frac{1}{m}\sum_{i=1}^{m}h_k(x;a_i,b_i)\qquad\hbox{for all }x\in\R^n.
\end{equation}
Thus, the penalty function in~\eqref{eq-penalized-function}
is written as
\begin{equation}\label{eq-penfun-hk}
	F_k(x) = f(x) + \g_kH_k(x)\qquad\hbox{for all }x\in\R^n.
\end{equation}
Regarding the average penalty $H_k(\cdot)$, we have an upper bound on its values at feasible points $x\in X$ via relation~\eqref{eq:hfunineq1}.
Specifically, setting $\d=\d_k$ in relation~\eqref{eq:hfunineq1}, we have $h_k(x; a_i,b_i) \le \frac{\d_k}{4\|a_i\|}\le \frac{\d_k}{4\a_{\min}}$  for all $x\in X$ and $i\in[m]$,
implying that 
\begin{equation}\label{eq-hk-atfeasx}
	H_k(x)\le \frac{\d_k}{4\a_{\min}}\qquad\hbox{for all $x\in X$ and all }k\ge1.
\end{equation}

To lower bound the value $H_k(x)$ at any $x\in\R^n$, we use
the following result regarding the intersection of linear sets.
\begin{lem}[Hoffman lemma~\cite{Hoffman1952}]\label{lem-hoffman}
	Given a collection of sets $X_i=\{x\in\R^n\mid \la a_i,x\ra -b\le 0\}$, $i\in[m]$,
	with a nonempty intersection $X=\cap_{i=1}^m X_i$,
	there exists a scalar
	$\beta=\beta(a_1,\ldots,a_m)>0$ such that
	$\beta \sum_{i=1}^{m} \dist(x, X_i) \ge \dist(x,X)$ for all  $x\in\R^n$. 
\end{lem}
The following lemma and its corollary provide some additional
properties of the penalty function $h_\d(\cdot;a,b)$ that we will use later on.
The proof can be found in Lemma~1 in~\cite{Penalty_siam}.

\begin{lem}\label{lem:penalty}
	Given a nonzero vector $a\in\R^n$ and a scalar $b\in\R$, 
	consider the penalty function $h_\d(\cdot;a,b)$ defined in~\eqref{eq:hfun} with $\d\ge0$.
	Let $Y=\{x\in\R^n\mid \la a,x\ra-b\le0\}.$ 
	Then, we have for $\d=0$, $h_0(x;a,b)=\dist(x,Y)$ for all $x\in\R^n$, 
	and for any $0\le \d \le \d'$,
	$h_\d(x;a,b) \le h_{\d'}(x;a,b)$ for all $x\in\R^n$.
\end{lem}

The following corollary is an immediate consequence of Lemma~\ref{lem:penalty},
which shows that any
feasible point $\hat x\in X$ can be used to construct non-empty level sets of 
$F_{\g\d}(\cdot)$ and $f(\cdot)$.
The proof can be found in Corollary~2 in~\cite{Penalty_siam}.
\begin{cor}\label{cor:lset}
	Let $\g>0$ and $\d\ge 0$ be arbitrary, and let
	$\hat x$ be a feasible point for the original problem~\eqref{eq:problem}.
	Then, for the scalar $t_{\g\d}(\hat x)$ defined by
	$t_{\g\d}(\hat x)= f(\hat x)+\g\d/(4\a_{\min}),$
	the level set
	$\{x\in\R^n\mid F_{\g\d}(x)\le t_{\g\d}(\hat x)\}$ is nonempty and
	$\{x\in\R^n\mid F_{\g\d}(x)\le t_{\g\d}(\hat x)\}\subseteq
	\{x\in\R^n\mid f(x)\le t_{\g\d}(\hat x)\}$.
	Moreover, the solution set $X^*_{\g\d}$ of
	the penalized problem~\eqref{eq:pen-problem} is contained in the level set
	$\{x\in\R^n\mid f(x)\le t_{\g\d}(\hat x)\}$.
\end{cor}
Now, a lower bound on the value $H_k(x)$ at any $x\in\R^n$ follows from
Hoffman Lemma~\ref{lem-hoffman} and Lemma~\ref{lem:penalty}, as follows.
By Lemma~\ref{lem:penalty} with $\d=0$, $\d'=\d_k$, $a=a_i$, $b=b_i$, and $Y=X_i$, we obtain $h_k(x;a_i,b_i)\ge h_0(x;a_i,b_i)=\dist(x,X_i)\quad\hbox{for all } x\in\R^n$ for all $i\in[m]$.
Therefore, for all $x\in\R^n$,
$H_k(x)=\frac{1}{m}\sum_{i=1}^{m} h_k(x; a_i,b_i)
\ge \frac{1}{m}\sum_{i=1}^{m} \dist(x,X_i)$.
By using Hoffman Lemma~\ref{lem-hoffman}, we obtain
\begin{equation}\label{eq-hk-atanyx}
	H_k(x)\ge \frac{1}{m\b} \dist(x,X)\quad\hbox{for all } x\in\R^n.
\end{equation}

The following result provides an upper bound on the distance of $x_k^*$ from the feasible set $X$,
and shows that this distance goes to 0 if $\g_k$ tends to infinity and $\d_k$ tends to 0. 

\begin{prop}\label{prop-distancetofeas}
	Let $f(\cdot)$ be continuous with bounded lower-level sets, 
	and let $\g_k>0$, $\d_k>0$, and $\g_k\d_k\le c$ for all $k$ and for some $c>0$.
	Then, for arbitrary $\hat x\in X$ and all $k$,
	\[\dist(x_k^*,X) \le \frac{m\b}{\g_k} (f(\hat x) - f(x_k^*))+\frac{m\b\d_k}{4\a_{\min}},\]
	where 
	$\beta$ is the Hoffman constant from Lemma~\ref{lem-hoffman}. In particular, 
	$\lim_{k\to\infty}\dist(x_k^*,X)=0$ as  $\g_k\to\infty$  and  $\d_k\to0$,
	with the convergence rate of the order $O(\g_k^{-1} +\d_k)$.
\end{prop}
We refer the reader to Appendix~\ref{app:distancetofeas} for the proof.

\subsection{Convex Function $f$ not Necessarily Differentiable}\label{subsec-var-par-b}
Here, we consider the case when $f(\cdot)$ is convex on $\R^n$. Since the domain of $f(\cdot)$ is $\R^n$, 
the function $f$ is continuous. The subdifferential set $\partial f(x)$ is nonempty at every $x\in\R^n$
since $f(\cdot)$ is convex and $\dom(f)=\R^n$. 
The function $f(\cdot)$ is strongly convex with a constant $\mu>0$ if and only if 
for all $u,v\in\R^n$ and all subgradients $\tilde\nabla f(x)\in\partial f(x)$, we have
\begin{equation}\label{eq-strconvex}
	f(v)+\la \tilde\nabla f(v),u-v\ra +\frac{\mu}{2}\|u-v\|^2\le f(u).
\end{equation}
Moreover, when $f(\cdot)$ is strongly convex with a constant $\mu$, we also have for all $x,y\in\R^n$, and all subgradients $\tilde\nabla f(x)\in\partial f(x)$
and $\tilde\nabla f(y)\in\partial f(y)$,
\begin{equation}\label{eq-subgrad-strong}
	\mu\|x-y\|^2\le \la \tilde \nabla f(x)-\tilde \nabla f(y),x-y\ra.
\end{equation}
If $f(\cdot)$ is just convex relations~\eqref{eq-strconvex} and~\eqref{eq-subgrad-strong} hold with $\m=0$.
To capture both cases when $f$ is strongly convex and when $f$ is just convex, 
we will sometimes abuse the definition of strong convexity in~\eqref{eq-strconvex} by allowing 
the possibility that $\mu=0$.

When $f(\cdot)$ is strongly convex with $\mu>0$, then so is every penalty function $F_k(\cdot)$ in~\eqref{eq-penalized-function} with the same $\mu>0$.
In this case, 
the original problem~\eqref{eq:problem} and each penalized problem $\min_{x\in\R^n}F_k(x)$, $k\ge1$, have unique solutions, respectively, denoted by $x^*\in X$ and $x_k^*\in\R^n$, respectively. 
Moreover, under mild conditions on the penalty parameters $\d_k$ and $\g_k$, the optimal set sequence $\{X_k^*\}$ is uniformly bounded in the case of continuous function with bounded lower-level sets, as seen in the following lemma. Its proof relies on Corollary~\ref{cor:lset} and is provided in Appendix~\ref{app:solsbded}.
\begin{lem}\label{lem-solsbded}
	Let $f(\cdot)$ be continuous and have bounded lower-level sets.
	Let $\g_k>0$, 
	$\d_k>0$, and $\g_k\d_k\le c$ for all $k\ge 1$ and for some $c>0$.
	Then, the optimal set $X_k^*$ for the penalized problem $\min_{x\in \R^n} F_k(x)$ 
	is nonempty compact set and
	$X_k^*\subseteq\left\{x\in\R^n\mid f(x)\le t_c(\hat x)\right\}$ for all~$k$, where 
	$t_c(\hat x)=f(\hat x)+c/(4\a_{\min})$, with $\hat x\in X$. 
	In particular, the set sequence $\{X_k^*\}$ is uniformly bounded. 
\end{lem}
\begin{rem}\label{remsgd-bound}
	When the conditions of Lemma~\ref{lem-solsbded} are satisfied, the optimal solutions of the penalized problems
	$\min_{x\in\R^n} F_k(x)$, for all $k\ge1$, 
	are uniformly bounded, i.e., there exists $D>0$ such that 
	$\|x_k^*\|\le D$ for all $x_k^*\in X_k^*$ and all $k\ge1$.
	Therefore, the projections $\Pi_X[x_k^*]$ of these optimal solutions on the feasible set $X$
	are also uniformly bounded, i.e., there exists $R>0$ such that 
	$\|\Pi_X[x_k^*]\|\le R$ for all $x_k^*\in X_k^*$ and all $k\ge1$.
	Hence, the subgradients $\tilde \nabla f(x)\in\partial f(x)$, for all $x$ with $\|x\|\le R$,
	are bounded, i.e., \begin{equation}\label{eq:sgdbound}
		L=\max_{\|x\|\le R}\{\|\tilde \nabla f(x)\|\,\mid \tilde \nabla f(x)\in\partial f(x)\}<\infty.
	\end{equation} 
\end{rem}

We next provide a set of conditions on parameters $\d_k$ and $\g_k$ ensuring
that the sequence $\{x_k^*\}$ converges to $x^*$ as $k\to\infty$ when $f(\cdot)$ is strongly convex. 
When $f(\cdot)$ is just convex, we obtain a special bound on $\dist(x_k^*,X)$ for any solution $x_k^*\in X_k^*$ to the penalized problem. 
This bound yields an improved convergence rate of $\dist(x_k^*,X)\to 0$ compared to 
that for a general function $f(\cdot)$ provided in Proposition~\ref{prop-distancetofeas}.

\begin{prop}\label{prop-sols}
	Let $f(\cdot)$ be strongly convex with $\mu\ge0$. If $\mu=0$, assume that $f(\cdot)$
	has bounded lower-level sets.
	Let $\g_k>0$, $\d_k>0$, and $\g_k\d_k\le c$ for all $k$ and some $c>0$. 
	Then, for all $k$, we have\vspace{-0.25cm}
	\begin{align*}
		&\frac{\mu}{2}\|x^* -x_k^*\|^2 + \frac{\mu}{2}\|x^*-\Pi_X[x_k^*]\|^2  +\left(\frac{\mu}{2}+\frac{\g_k}{m\b} -L\right)\dist(x_k^*,X)
		\le \frac{\g_k\d_k}{4\a_{\min}},
	\end{align*}
	where $L$ is given by~\eqref{eq:sgdbound}.
\end{prop}
The proof is in Appendix~\ref{app:sols}.

Proposition~\ref{prop-sols} indicates that, when $\mu>0$,
by letting $\g_k\to\infty$,
we will have $\frac{\g_k}{4m\b}\ge L$ for all large enough $k$,
implying\vspace{-0.25cm}
\begin{align}\label{eq-disttofeas-convex}
	\dist(x_k^*,X)&\le \frac{\g_k\d_k}{4\a_{\min} \left(\frac{\mu}{2}+\frac{\g_k}{4m\b}-L\right)}\approx O(\d_k).
\end{align}
Thus, if $\d_k\to0$,  the distance of $x_k^*$ to the feasible set $X$ will go to 0 asymptitically
at the rate of  $O(\d_k)$, independent of $\g_k$. 
The preceding relation holds with $\mu=0$ if $f(\cdot)$ is merely convex. 
The convergence rate order is better than $O(\g_k^{-1})+O(\d_k)$ obtained for a general $f(\cdot)$ 
in Proposition~\ref{prop-distancetofeas}.

When $f(\cdot)$ is strongly convex with $\mu>0$, 
Proposition~\ref{prop-sols} shows that $\|x^* -x_k^*\|^2\le  \frac{\g_k\d_k}{2\mu\a_{\min}}$ for large enough $k$.
If additionally $\g_k\d_k\to0$, then the solutions $x_k^*$ of the penalized problems converge to
the optimal solution $x^*$ of the original problem, with the rate in the order of $O(\g_k\d_k)$. 

The common requirement for $\dist(x_k^*,X)\to0$, as $k\to\infty$, in Proposition~\ref{prop-distancetofeas} and Proposition~\ref{prop-sols} is that $\d_k\to0$.
The main difference between these propositions is in the requirement for the penalty parameter $\g_k$. 
Specifically,
to ensure $\dist(x_k^*,X)\to0$, the penalty $\g_k$ has to increase to $+\infty$
for a general function~$f(\cdot)$ (Proposition~\ref{prop-distancetofeas}). 
In contrast, 
to ensure that $\dist(x_k^*,X)\to0$ for a convex function $f(\cdot)$, one can choose a fixed penalty value $\g_k=\g$, for all $k$, with $\gamma$ large enough so that $\frac{\g}{4m\b}> L$, as seen from relation~\eqref{eq-disttofeas-convex}. 
However, determining such a value of $\g$ is challenging as it is hard to obtain upper estimates for the subgradient norm bound $L$ and the Hoffman constant $\b$. Estimating $L$ requires knowing a region that contains the projections of the solutions to the penalized problems on the feasible set (see~\eqref{eq:sgdbound}). Determining an upper estimate of the Hoffman constant $\b$ is also a difficult problem, which has recently been addressed in~\cite{pena2021} via computational approaches.

Based on Proposition~\ref{prop-sols}, one can construct a two-loop iterative approach to compute
the optimal point $x^*$ of the original problem {in the case of strongly convex $f$}. The outer loop is on the index $k$ where the penalty values $\g_k$ and $\d_k$ are set.
For any given $k$, the inner loop of iterations compute the optimal point $x_k^*$ for the penalized problem $\min_{x\in \R^n}F_k(x)$. 
This naive two-loop approach is quite inefficient. Later in Section~\ref{sec-algo}, we propose a more efficient single-loop  algorithm,
where at each iteration $k$, we 
adjust the parameters $\g_k$ and $\d_k$, and use a (stochastic) gradient of the penalty function $F_k(\cdot)$ for the update.

\section{Random Incremental Penalty Algorithm}\label{sec-algo}
Assuming that the function $f(\cdot)$ is convex over $\R^n$,
we consider an algorithm that takes one gradient step for minimizing $F_k(\cdot)$ at iteration $k$,
as opposed to determining $x^*_k$ for each function $F_k(\cdot)$.
To deal with the large number of component functions $h_k(\cdot;a_{i},b_{i})$ involved in
$F_k(\cdot)$, we consider a random incremental subgradient algorithm using only 
one randomly chosen constraint (indexed by $i_k$) to estimate
a subgradient $\tilde \nabla F_k(x_k)\in\partial F_k(x_k)$ at iteration $k$, when $x_k$ is available.
This estimation is employed to construct $x_{k+1}$ via {\textit {random penalty}} corresponding to subgradient sampling,
as opposed to determining a full subgradient of $F_k(x_k)$.

We apply random incremental update to the penalty function~\eqref{eq-penalized-function} 
represented in the following form:
\[F_k(x) = \frac{1}{m} \sum_{i=1}^m\left( f(x) + \g_k  h_k\left(x; a_i, b_i\right) \right),\]
and the random incremental penalty method is:
for $k\ge 1$,
\begin{equation}\label{eq:gradmet}
	x_{k+1} = x_k-s_k [\tilde \nabla f(x_k) + \g_k\nabla h_k(x_k;a_{i_k},b_{i_k})],
\end{equation}
where $s_k>0$ is a stepsize, $\tilde \nabla f(x_k)$ is a subgradient of $f(\cdot)$ at $x=x_k$, 
and the index $i_k\in\{1,\ldots,m\}$ is chosen uniformly at random at every iteration $k$.
\an{The algorithm is initiated with a random initial point $x_1\in\R^n$, for which we assume that 
	$\ex{\|x_1\|^2}<\infty$.} 

Note that $\tilde \nabla f(x_k) + \g_k\nabla h_k(x_k;a_{i_k},b_{i_k})$  is an unbiased estimate of 
a subgradient $\tilde \nabla F_k(x_k) \in\partial F_k(x_k)$, 
since by the uniform distribution of $i_k$ we have
\begin{align}\label{eq-nobias}
	\ex{\tilde \nabla f(x_k) + \g_k\nabla h_k(x;a_{i_k},b_{i_k})|\EuScript F_k}=\tilde \nabla F_k(x_k),
\end{align}
where $\EuScript F_k$ is the $\sigma$-algebra generated by the random variables 
$\{i_j,\, 1 \le j\le k-1\}$ and the random initial iterate $x_1$, which is equivalent to 
the $\sigma$-algebra $\EuScript F_k = \sigma(\{x_t\}_{t=1}^{k})$ for all $k\ge 1$.

Unlike the standard random incremental method, 
the random incremental penalty method in~\eqref{eq:gradmet} selects one 
random component from the time-varying function $F_k(\cdot)$.
This makes the analysis of the method more challenging since the iterates $\{x_k\}$ need not be feasible for the original problem,
which poses difficulties, and the penalty parameters $\g_k$ and $\d_k$ have to be carefully tuned
to ensure the convergence of the iterates.

\subsection{Preliminary results}

We first establish a basic result for the iterates $x_k$ of the random method~\eqref{eq:gradmet}.
\begin{lem}\label{lem-basic-iter}
	Let $f(\cdot)$ be strongly convex with $\mu\ge 0$, and 
	let $\g_k>0$, $\d_k>0$, and $s_k>0$ for all $k$.
	Then, the iterates $x_k$ of the method~\eqref{eq:gradmet} surely satisfy 
	for all $y\in X$ and $k\ge1$,
	\begin{align*}
		&\|x_{k+1}- y\|^2 \le(1-\mu s_k)\|x_k - y\|^2 
		+2s_k(f(y)-f(x_k))\cr
		&+\frac{s_k\g_k \d_k}{2\a_{\min}} - 2s_k\g_k\dist(x_k,X_{i_k})
		+s_k^2(\|\tilde\nabla f(x_k)\| + \g_k)^2.
	\end{align*}
\end{lem}

{\it Proof}:
By the definition of $x_{k+1}$ in~\eqref{eq:gradmet}, we surely have 
for any $y\in X$ and all $k\ge1$,
\begin{align*}
	\|x_{k+1}- y\|^2 =\|x_k - y\|^2 &-2s_k\la g_k(x_k),x_k-y\ra +s_k^2\| g_k(x_k)\|^2,
\end{align*}
with $g_k(x_k)=\tilde\nabla f(x_k) + \g_k\nabla h_k(x_k;a_{i_k},b_{i_k})$.
By the convexity of $h_k(\cdot;a_{i_k},b_{i_k})$, we have
for all $y\in X$ and all $k\ge1$,
\begin{align*}
	&\|x_{k+1}- y\|^2 \le\|x_k - y\|^2 
	-2s_k\la \tilde\nabla f(x_k), x_k - y\ra \cr
	&+2s_k \g_k(h_k(y;a_{i_k},b_{i_k})-h_k(x_k;a_{i_k},b_{i_k}))
	+s_k^2\| g_k(x_k)\|^2.
\end{align*}
Since $y$ is feasible, we have $h_k(y;a_{i_k},b_{i_k})\le \d_k / 4\a_{\min}$, with $\a_{\min}=\min_{i\in[m]}\|a_i\|$,
which follows by relation~\eqref{eq:hfunineq1} (where $\d=\d_k$ and $x=y$).
By the monotonicity of $h_\d(\cdot;a_i,b_i)$ with respect to $\d$, we have that
$\dist(x,X_i)=h_0(x;a_{i},b_{i})\le h_{\d_k}(x;a_i,b_i)$ for all $x$ and $i$
(see Lemma~\ref{lem:penalty} where $\d=0$, $\d'=\d_k$).
Hence, it follows that for all $y\in X$ and all $k\ge1$,
\begin{align*}
	&\|x_{k+1}- y\|^2 \le\|x_k - y\|^2 
	-2s_k\la \tilde\nabla f(x_k), x_k - y\ra\cr
	&+\frac{s_k\g_k \d_k}{2\a_{\min}} - 2s_k\g_k\dist(x_k,X_{i_k})
	+s_k^2\| g_k(x_k)\|^2.
\end{align*}
By the strong convexity relation (where $u=y$, $v=x_k$), it follows that 
\[-2s_k\la \tilde\nabla f(x_k), x_k - y\ra\le 2s_k(f(y)-f(x_k))-\mu s_k\|x_k-y\|^2.\]
By combining the preceding two relations,
we obtain that, surely, for all $y\in X$ and all $k\ge1$,
\begin{align*}
	&\|x_{k+1}- y\|^2 \le(1-\mu s_k)\|x_k - y\|^2 
	+2s_k(f(y)-f(x_k))\cr
	&+\frac{s_k\g_k \d_k}{2\a_{\min}} - 2s_k\g_k\dist(x_k,X_{i_k})
	+s_k^2\| g_k(x_k)\|^2.
\end{align*}
To estimate $\| g_k(x_k)\|$, we write
\[\|g_k(x_k)\|\le \|\tilde \nabla f(x_k)\| + \g_k\|\nabla h_k(x_k;a_{i_k},b_{i_k})\|
\le \|\tilde \nabla f(x_k)\| + \g_k,\]
where the last inequality uses the fact that 
$\|\nabla h_k(x;a_{i},b_{i})\|\le 1$ for all $x$ and all $i$ (see Lemma~\ref{lem:pderiv}).
The stated relation follows from the preceding two inequalities.
\qed

Lemma~\ref{lem-basic-iter} is important for both convergence and convergence rate analysis.
We now introduce additional assumptions and refine Lemma~\ref{lem-basic-iter}.
Specifically, we assume that the objective function is convex and has bounded level sets, and the subgradient norms 
$\|\tilde \nabla f(x)\|$ grow at most linearly with $\|x\|$.
\begin{assum}\label{asum:fconvex}
	The function $f(\cdot)$ is convex and has bounded level sets.
\end{assum}
Note that if $f(\cdot)$ is strongly convex with $\mu>0$, then Assumption~\ref{asum:fconvex} is satisfied. 
The assumption is also satisfied if $f(\cdot)$ is convex and coercive, i.e., $\lim_{\|x\|\to\infty}f(x)=+\infty$.	 
We make the following assumption regarding the subgradients of $f(\cdot)$. 
\begin{assum}\label{assum-infty}
	There exist scalars $M_1,M_2>0$
	such that $\|\tilde \nabla f(x)\|\le M_1\|x\|+M_2$
	for all subgradients $\tilde \nabla f(x)\in\partial f(x)$ and for $x\in\R^n$.
\end{assum}

The following lemma will be important in establishing the convergence properties of the method.
\begin{lem}\label{lem-key}
	Let Assumption~\ref{asum:fconvex} and Assumption~\ref{assum-infty} hold.
	Then, it surely holds
	for all $x^*\in X^*$  and $k\ge1$,
	\begin{align*}
		&\|x_{k+1}- x^*\|^2 \le(1+4s_k^2M_1^2)\|x_k - x^*\|^2 \cr
		&\ \ 
		+2s_k\left((1-\rho) B +\rho( M_1B+M_2)\right)\dist(x_k,X)\cr
		&\ \ 
		+2s_k\rho(f^*- f(\Pi_X[x_k])) 
		+2s_k\rho M_1\|x_k-x^*\|\,\dist(x_k,X)\cr
		& \ \ +\frac{s_k\g_k \d_k}{2\a_{\min}}- 2s_k\g_k\dist(x_k,X_{i_k})
		+4s_k^2(M_1^2 B^2 +M_2^2+\g_k^2),\qquad
	\end{align*}
	where $\rho\in[0,1]$ is arbitrary scalar and $B$ is a norm bound for the optimal solutions $x^*\in X^*$ and the subgradients $\tilde\nabla f(x^*)$ at any $x^*\in X^*$.
\end{lem}

{\it Proof}:
By Lemma~\ref{lem-basic-iter}, where we omit the term $-\mu s_k$ in the coefficient of
$\|x_k - y\|^2$, we surely have
for all $y\in X$ and $k\ge1$,
\begin{align}\label{eq-jedan}
	&\|x_{k+1}- y\|^2 \le \|x_k - y\|^2 
	+2s_k(f(y)-f(x_k))\cr
	&+\frac{s_k\g_k \d_k}{2\a_{\min}} - 2s_k\g_k\dist(x_k,X_{i_k})
	+s_k^2(\|\tilde\nabla f(x_k)\| + \g_k)^2.
\end{align}
Under our assumption on the subgradient-norm growth (Assumption~\ref{assum-infty})
we have that 	
\[\|\tilde\nabla f(x_k)\| + \g_k
\le M_1\|x_k\|+M_2 + \g_k
\le M_1(\|x_k-y\|+\|y\|)+M_2 + \g_k\]
where the last inequality is obtained by using $\|x_k\|\le \|x_k-y\|+\|y\|$.
Next, using $(a+b+c+d)^2\le 4 (a^2 + b^2+c^2+d^2)$, which is valid for any scalars $a,b,c,$ and $d$,
we have
\[(\|\tilde\nabla f(x_k)\| + \g_k)^2\le 4 M_1^2(\|x_k-y\|^2+\|y\|^2)+4M_2^2+4\g_k^2.\]
Hence, by substituting the preceding estimate back in relation~\eqref{eq-jedan},
we surely have for all $y\in X$ and all $k\ge1$,
\begin{align}\label{eq-rel-perk}
	&\|x_{k+1}- y\|^2 \le(1+4s_k^2 M_1^2)\|x_k - y\|^2 
	+2s_k(f(y)- f(x_k)) \cr
	& +\frac{s_k\g_k \d_k}{2\a_{\min}}- 2s_k\g_k\dist(x_k,X_{i_k})
	+4s_k^2(M_1^2 \|y\|^2 +M_2^2+\g_k^2).\qquad
\end{align}
By our assumption that the function $f(\cdot)$ has bounded level sets (Assumption~\ref{asum:fconvex}),
the problem $\min_{x\in X}f(x)$ has a nonempty compact convex solution set $X^*$, i.e.,  $\|x^*\|\le B$
for some $B>0$ and for all $x^*\in X^*$.
Let $x^*\in X^*$ be an arbitrary solution,
and let $y=x^*$ in~\eqref{eq-rel-perk}. Thus, we  surely have for all $x^*\in X^*$  and $k\ge1$,
\begin{align}\label{eq-rel-perk-last}
	&\|x_{k+1}- x^*\|^2 \le(1+4s_k^2M_1^2)\|x_k - x^*\|^2 
	+2s_k(f^*- f(x_k)) \cr
	& +\frac{s_k\g_k \d_k}{2\a_{\min}}- 2s_k\g_k\dist(x_k,X_{i_k})
	+4s_k^2(M_1^2 B^2 +M_2^2+\g_k^2),\qquad
\end{align}
where $f^*$ is the optimal value of the original problem.

The remaining part of the proof relies on using two different ways to upper bound the value $f^*- f(x_k)$ in~\eqref{eq-rel-perk-last} by using the convexity of $f(\cdot)$. One way is to 
write  for any $x^*\in X^*$,
\begin{eqnarray}\label{eq:lowerbound}
	f^*- f(x_k)
	&=& f(x^*)- f(x_k)\cr 
	&\le &\la\tilde\nabla f(x^*),x^*-x_k\ra \cr
	&=&\la\tilde\nabla f(x^*),x^*-\Pi_X[x_k]\ra +\la\tilde\nabla f(x^*),\Pi_X[x_k]-x_k\ra \cr
	&\le& \|\tilde\nabla f(x^*)\|\, \|\Pi_X[x_k]-x_k\|,
\end{eqnarray}
where the last inequality is obtained by using $\la\tilde\nabla f(x^*),x^*-\Pi_X[x_k]\ra\le0$,
which holds by the optimality of $x^*$ (since $\Pi_X[x_k]$ is feasible),
and by applying the Cauchy-Schwarz inequality to estimate the other inner product term.
Since $X^*$ is bounded by Assumption~\ref{asum:fconvex}, we may assume without loss of generality that $B$ is large enough so that
$\|\tilde\nabla f(x^*)\|\le B$ and $\|x^*\|\le B$ for all subgradients $\tilde\nabla f(x^*)$ and all $x^*\in X^*$, so that we surely have
\begin{equation}\label{eq-oneway}
	f^*- f(x_k)\le B \|\Pi_X[x_k]-x_k\|
\end{equation}	

Another way is to write
\begin{eqnarray*}
	f^*- f(x_k) &=&
	f^*-f(\Pi_X[x_k])+f(\Pi_X[x_k]) -f(x_k)\cr
	&\le &f^*-f(\Pi_X[x_k])+
	\la\tilde\nabla f(\Pi_X[x_k]),\Pi_X[x_k]-x_k\ra\cr
	&\le & f^*-f(\Pi_X[x_k])+
	\|\tilde\nabla f(\Pi_X[x_k])\|\,\|\Pi_X[x_k]-x_k\|.
\end{eqnarray*}
Using the assumption that the subgradient-norm growth is at most linear 
(Assumption~\ref{assum-infty}),
we have that 
\[\|\tilde\nabla f(\Pi_X[x_k])\|\le M_1\|\Pi_X[x_k]\| +M_2\le M_1(\|\Pi_X[x_k]-x^*\| +B)+M_2,\]
where we use the fact that $\|x^*\|\le B$
for any $x^*\in X^*$.
By combining the preceding two relations and using 
the fact that $\|\Pi_X[x_k]-x^*\|\le \|x_k-x^*\|$,
we obtain
\begin{equation}\label{eq-anotherway}
	f^*- f(x_k) \le  f^*-f(\Pi_X[x_k]) +
	(M_1\|x_k-x^*\| +M_1B+M_2 )\|\Pi_X[x_k]-x_k\|
\end{equation}
Multiplying the inequality~\eqref{eq-oneway} with $1-\rho$ and the inequality~\eqref{eq-anotherway} with $\rho$, for some $\rho\in[0,1]$, and combining these 
with inequality~\eqref{eq-rel-perk-last}, we obtain surely 
for all $x^*\in X^*$  and $k\ge1$,
\begin{align*}
	&\|x_{k+1}- x^*\|^2 \le(1+4s_k^2M_1^2)\|x_k - x^*\|^2 +2s_k(1-\rho)
	B \|\Pi_X[x_k]-x_k\|\cr
	&\ \ 
	+2s_k\rho(f^*- f(\Pi_X[x_k])) 
	+2s_k\rho\left(M_1\|x_k-x^*\| +M_1B+M_2\right)\,\|\Pi_X[x_k]-x_k\|\cr
	& \ \ +\frac{s_k\g_k \d_k}{2\a_{\min}}- 2s_k\g_k\dist(x_k,X_{i_k})
	+4s_k^2(M_1^2 B^2 +M_2^2+\g_k^2),\qquad
\end{align*}
which gives the stated relation by noting that 
and
$\|\Pi_X[x_k]-x_k\|=\dist(x_k,X)$, and by grouping the terms accordingly.
\qed

\subsection{Almost Sure and in-Expectation Convergence}
In this section, we establish almost sure convergence of the random incremental penalty method~\eqref{eq:gradmet}.
In the forthcoming discussion we will often use 
\as\, for
{\textit {almost surely}}.
Our convergence analysis makes use of
the following result on semi-supermartingale convergence,
which is due to Robbins and Siegmund~\cite{Robbins1971} (it can also be found in~\cite{polyak87}, Lemma 11, page 50).

\begin{lem}\cite{Robbins1971}\label{lem-polyak}
	Let $\{v_k\}$, $\{u_k\}$, $\{\a_k\}$, and $\{\b_k\}$ be random nonnegative scalar sequences such that 
	$\sum_{k=0}^\infty \a_k<\infty$ and $\sum_{k=0}^\infty \b_k<\infty$ \as,
	and
	\[\mathbb{E}\left[v_{k+1}\mid \mathcal{F}_k \right]\le(1+\a_k) v_k -u_k+\b_k\quad \hbox{for all }k\geq 0\ \as,\]
	where 
	$\mathcal{F}_k=\{v_\ell,u_\ell,\a_\ell,\b_\ell;\, 0\le \ell\le k\}$.
	Then, we have \as\ that $\sum_{k=0}^{\infty}u_k<\infty$ and $\lim_{k\to\infty}v_k=v$ for a random variable $v\geq 0$.
\end{lem}

Using Lemma~\ref{lem-polyak}, we establish the following result which will be used in the convergence analysis of the method. 
\begin{lem}\label{lem-opt}
	Consider a minimization problem $\min_{x \in Z} \phi(z)$, where $\phi:\mathbb{R}^n\to\mathbb{R}$ is a 
	continuous function and $Z\subseteq\mathbb{R}^n$ is a closed convex set. Assume that 
	the solution set $Z^*$ of the problem is nonempty. Let 
	$\{z_k\}\subset \mathbb{R}^n$ be a random sequence, 
	$\{a_k\}$ and $\{c_k\}$ be random nonnegative scalar sequences, and $\{b_k\}$ and  $\{b'_k\}$ be deterministic nonnegative scalar sequences
	such that \as\ for all $z^*\in  Z^*$ and for all $k\ge1,$
	\[\mathbb{E}[\|z_{k+1}-z^*\|^2\mid\mathcal{F}_k] \le (1+a_k)\|z_k-z^*\|^2 
	- b_k\left (\phi(\Pi_Z[z_k]) -\phi^*\right) +c_k,\]
	\[\sum_{k=1}^\infty b'_k\dist(z_k,Z_k)<\infty\qquad\as\]
	where $\mathcal{F}_k=\{z_1,\ldots,z_k\}$ for all $k$,
	$\phi^*=\min_{x \in Z} \phi(z)$, while the scalar sequences satisfy $\sum_{k=1}^\infty a_k<\infty$ and
	$\sum_{k=1}^\infty c_k<\infty$ \as, and
	$\sum_{k=1}^\infty b_k=\infty$ and $\sum_{k=1}^\infty b'_k=\infty$. 
	Then, the sequence $\{z_k\}$ converges to some (random)  optimal solution $z^*\in X^*$ almost surely. Moreover, if the solution set $Z^*$ is bounded, then there exists a scalar $M_0$ such that $\|z_k\|\le M_0$ for all $k\ge0$ \as and convergence to $z^*$ is also in expectation. 
\end{lem}

{\it Proof}:
We note that the conditions of Lemma~\ref{lem-polyak} are satisfied with $v_k=\|z_k-z^*\|^2$, for every $z^*\in Z^*$, and $u_k=b_k\left (\phi(\Pi_Z[z_k]) -\phi^*\right)$, and by this lemma
we obtain the following statements:
\begin{equation}\label{eq:xtconv}
	\hbox{$\{\|z_k - z^*\|^2\}$ converges \as\ for each $z^*\in Z^*$},
\end{equation}
\begin{equation}\label{eq:sumf}
	\sum_{k=1}^\infty b_k \left(\phi(\Pi_Z[z_k]) - \phi^* \right)<\infty \qquad \as.\end{equation}  
Since $\sum_{k=1}^\infty b_k=\infty$, it follows from~\eqref{eq:sumf} that \as,
$\liminf_{k\to\infty} \phi(\Pi_Z[z_k])=\phi^*$. 
The conditions $\sum_{k=1}^\infty b'_k\dist(z_k,Z_k)<\infty$ \as, and $\sum_{k=1}^\infty b'_k=\infty$
imply that 
$\liminf_{k\to\infty} \|z_k-\Pi_Z[z_k]\|=0.$

Let $\{z_{k_\ell}\}$ be a subsequence of $\{z_k\}$ such that \as,
\begin{equation}\label{eq-lim1}
	\lim_{\ell\to\infty} \phi(\Pi_Z[z_{k_\ell}])
	=\liminf_{k\to\infty} \phi(\Pi_Z[z_k])=\phi^*,\end{equation} 
\begin{equation}\label{eq-lim2}
	\lim_{\ell\to\infty} \|z_{k_\ell}-\Pi_Z[z_{k_\ell}]\| =\liminf_{k\to\infty} \|z_k-\Pi_Z[z_k]\|=0.
\end{equation} 
Now, relation~\eqref{eq:xtconv} implies that the sequence $\{z_k\}$ is bounded \as, so  
without loss of generality, we can assume that $\{z_{k_\ell}\}$ is converging \as\ to some 
random point $\tilde z$ (for otherwise, we
can in turn select \as\ convergent subsequence of  $\{z_{k_\ell}\}$). Since the projection mapping $z\mapsto \Pi_Z[z]$ is continuous, it follows that 
\[\lim_{\ell\to\infty}\Pi_Z[z_{k_\ell}]
=\Pi_Z[\tilde z]\qquad\as\]
The preceding relation and relation~\eqref{eq-lim2}
imply that 
\[0=\lim_{\ell\to\infty}\|z_{k_\ell}-\Pi_Z[z_{k_\ell}]\|=
\|\tilde z -\Pi_Z[\tilde z]\|\qquad\as\]
Therefore, we have that $\tilde z\in Z$ \as
Moreover, by the continuity of $\phi(\cdot)$,
\[\lim_{\ell\to\infty} \phi(z_{k_\ell})=\phi(\tilde z)\qquad\as,\]
which by relation~\eqref{eq-lim1} and the fact $\tilde z\in Z$ \as\ implies that $\tilde z\in Z^*$ \as\ By letting $z^*=\tilde z$ in~\eqref{eq:xtconv}
we obtain that $\{z_k\}$ converges to $\tilde z$ \as

When the set $Z^*$ is bounded, the convergence point $\tilde z$ of the iterates $\{z_k\}$ is bounded by a (deterministic) scalar almost surely, implying that $\{\|z_k\|\}$ is \as\  bounded by some deterministic scalar $M_0$. Finally, using  Lebesgue's dominated convergence theorem, we conclude that $\{z_k\}$ converges to $\tilde z$ also in expectation.

\qed

Having Lemma~\ref{lem-key}, 
Lemma~\ref{lem-polyak}, and Lemma~\ref{lem-opt} in place, we  show next the convergence of the method
under some conditions on the stepsize $s_k$ and the penalty parameters $\g_k,\d_k$, as given in the following assumption.
\begin{assum}\label{assum-parameters} 
	Let $\g_k>0$, $\d_k>0$, and $s_k>0$ for all $k$, and assume that
	$\lim_{k\to\infty}\g_k=\infty$, $\sum_{k=1}^\infty s_k=\infty$, 
	$\sum_{k=1}^\infty s_k\g_k \d_k<\infty$, $\sum_{k=1}^\infty s_k^2\g_k^2<\infty$. 
\end{assum}

The intuition behind the conditions in Assumption~\ref{assum-parameters} is as follows. The condition $\lim_{k\to\infty}\g_k=\infty$ \an{ensures that the penalty function pushes the iterations} into the feasible set as time runs, whereas $\sum_{k=1}^\infty s_k=\infty$ allows the algorithm to make sufficient progress toward an optimal solution of the original problem. The last two conditions, $\sum_{k=1}^\infty s_k\g_k \d_k<\infty$ and $\sum_{k=1}^\infty s_k^2\g_k^2<\infty$, keep the perturbations caused by the penalty function under control.

In the following proposition, we establish almost sure iterates' convergence
and their boundedness properties. For this purpose we define the sigma-algebra $\EuScript F_k$ relevant to the random method~\eqref{eq:gradmet} as follows:
\[\EuScript F_k=\{x_1,\ldots,x_k\}\qquad\hbox{for all }k\ge 1.\]

\begin{prop}\label{prop-convergence}
	Let Assumptions~\ref{asum:fconvex}--\ref{assum-parameters} hold. Then, the iterates of the random method~\eqref{eq:gradmet} converge to a (random)  optimal solution \as\ and in expectation. Moreover, there exists a scalar $M>0$ such that
	$\|x_k\|\le M$ for all $k\ge 1$ \as
\end{prop}

{\it Proof}: We use Lemma~\ref{lem-key}, according to which we surely have for all $x^*\in X^*$, all $\rho\in[0,1]$, and all $k\ge 1$,
\begin{align*}
	&\|x_{k+1}- x^*\|^2 \le(1+4s_k^2M_1^2)\|x_k - x^*\|^2 \cr
	&\ \ 
	+2s_k\left((1-\rho) B +\rho( M_1B+M_2)\right)\dist(x_k,X)\cr
	&\ \ 
	+2s_k\rho(f^*- f(\Pi_X[x_k])) 
	+2s_k\rho M_1\|x_k-x^*\|\,\dist(x_k,X)\cr
	& \ \ +\frac{s_k\g_k \d_k}{2\a_{\min}}- 2s_k\g_k\dist(x_k,X_{i_k})
	+4s_k^2(M_1^2 B^2 +M_2^2+\g_k^2),\qquad
\end{align*}
We will take the conditional expectation in the preceding relation with respect to $\EuScript F_k=\{x_1,\ldots,x_k\}$. In doing so, we note that, given $x_k$, the index
$i_k$ is uniformly distributed over $\{1,\ldots,m\}$, for all $k\ge1$, implying that
\[\ex{\dist(x_k,X_{i_k})\mid \EuScript F_k} =\frac{1}{m}\sum_{i=1}^m\dist(x_k,X_i).\]
By using Hoffman's Lemma (Lemma~\ref{lem-hoffman}), we obtain
\[\ex{\dist(x_k,X_{i_k})\mid \EuScript F_k}\ge \frac{1}{\b m}\,\dist(x_k,X).\]
Therefore, it follows that \as\ for all $x^*\in X^*$, all $\rho\in[0,1]$, and all $k\ge 1$, 
\begin{align}\label{eq-keys}
	&
	\ex{\|x_{k+1}- x^*\|^2 \mid \EuScript F_k}\le(1+4s_k^2M_1^2)\|x_k - x^*\|^2 \cr
	&\ \ 
	+2s_k\left((1-\rho) B +\rho( M_1B+M_2)\right)\dist(x_k,X)\cr
	&\ \ 
	+2s_k\rho(f^*- f(\Pi_X[x_k])) 
	+2s_k\rho M_1\|x_k-x^*\|\,\dist(x_k,X)\cr
	& \ \ +\frac{s_k\g_k \d_k}{2\a_{\min}}- \frac{2s_k\g_k}{\b m}\,\dist(x_k,X)
	+4s_k^2(M_1^2 B^2 +M_2^2+\g_k^2).\qquad
\end{align}

The proof proceeds in two major steps both using relation~\eqref{eq-keys}: (Step 1) we show that 
$\sum_{k=1}^\infty s_k\dist(x_k,X)<\infty$ \as, and (Step 2) we show the almost sure convergence of the method by means of Lemma~\ref{lem-opt}.

\noindent {\it Step~1}:
In relation~\eqref{eq-keys}, we let $x^*\in X^*$ be arbitrary but fixed, and we set 
$\rho=0$. In the resulting relation, we group the common coefficients with 
$\dist(x_k,X)$ and obtain \as\ for all $k\ge 1$,
\begin{align*}
	&
	\ex{\|x_{k+1}- x^*\|^2 \mid \EuScript F_k}\le(1+4s_k^2M_1^2)\|x_k - x^*\|^2 
	+\frac{s_k\g_k \d_k}{2\a_{\min}}\cr
	&\ \ - 2s_k\left(\frac{\g_k}{\b m} - B\right)\,\dist(x_k,X)
	+4s_k^2(M_1^2 B^2 +M_2^2+\g_k^2).\qquad
\end{align*}
By Assumption~\ref{assum-parameters}, we have that $\lim_{k\to\infty}\g_k=\infty$, implying that 
for some sufficiently large $k_1\ge1$, we have $\frac{\g_k}{m\b} - B\ge B$. Therefore,
for all $k\ge k_1$ \as,
\begin{align*}
	&\ex{\|x_{k+1}- x^*\|^2 \mid \EuScript F_k}\le(1+4s_k^2M_1^2)\|x_k - x^*\|^2 
	+\frac{s_k\g_k \d_k}{2\a_{\min}}\cr	
	&\ \  - 2s_kB\dist(x_k,X)
	+4s_k^2(M_1^2B^2 +M_2^2+\g_k^2).
\end{align*} 
Under Assumption~\ref{assum-parameters},
the preceding relation
satisfies all the conditions of Lemma~\ref{lem-polyak}, starting at time $k_1$, 
with $v_k=\|x_k - x^*\|^2$,
$\a_k=4s_k^2M_1^2$, 
$u_k=2s_kB\dist(x_k,X)$, and 
$\b_k=\frac{s_k\g_k \d_k}{2\a_{\min}} +4s_k^2(M_1^2B^2 +M_2^2+\g_k^2)$.
Thus, by Lemma~\ref{lem-polyak}, 
\begin{equation}\label{eq-step1}
	\sum_{k=1}^\infty s_k\dist(x_k,X)<\infty,\qquad\as 
\end{equation}

\noindent	
{\it Step~2}:
We now use relation~\eqref{eq-keys} with $\rho=1$
and obtain  that \as\ for all $x^*\in X^*$ and all $k\ge 1$, 
\begin{align*}
	&\ex{\|x_{k+1}- x^*\|^2 \mid \EuScript F_k}\le(1+4s_k^2M_1^2)\|x_k - x^*\|^2 \cr
	&\ \ +2s_k(f^*- f(\Pi_X[x_k])) 
	+2s_k M_1\|x_k-x^*\|\,\dist(x_k,X)\cr
	& \ \ +\frac{s_k\g_k \d_k}{2\a_{\min}}- 2s_k\left(\frac{\g_k}{\b m}-(M_1B+M_2)\right)\,\dist(x_k,X)
	+4s_k^2(M_1^2 B^2 +M_2^2+\g_k^2).\qquad
\end{align*}
By Assumption~\ref{assum-parameters}, we have that $\lim_{k\to\infty}\g_k=\infty$. Thus, 
for some sufficiently large $k_2\ge1$, we have $\frac{\g_k}{m\b} - (M_1B+M_2)\ge 0$ and $M_1^2 B^2 +M_2^2\le \g_k$ for all $k\ge k_2$. Using this and 
noting that $2\|x_k-x^*\|\le 1+\|x_k-x^*\|^2$, we find that \as\ for all $x^*\in X^*$ and all $k\ge k_2$, 
\begin{align}\label{eq-keys1}
	\ex{\|x_{k+1}- x^*\|^2 \mid \EuScript F_k}&\le\left(1+4s_k^2M_1^2+s_k M_1\,\dist(x_k,X)\right)
	\|x_k - x^*\|^2 \cr
	&\ \ +2s_k(f^*- f(\Pi_X[x_k])) 
	+s_k M_1\,\dist(x_k,X) \cr
	&\ \ +\frac{s_k\g_k \d_k}{2\a_{\min}}
	+8s_k^2\g_k^2.\qquad
\end{align}
By Assumption~\ref{assum-parameters} we have
$\sum_{k=1}^\infty s_k=\infty$.
By Assumption~\ref{assum-parameters} 
and the fact that $\sum_{k=1}^\infty s_k\dist(x_k,X)<\infty$ \as\ (see ~\eqref{eq-step1}),
it follows that relation~\eqref{eq-keys1}
satisfies the conditions of Lemma~\ref{lem-opt}
with $z_k=x_k$, $f(\cdot)=\phi(\cdot)$, and
\[a_k=4s_k^2M_1^2+s_k M_1\,\dist(x_k,X),
\qquad
b_k=2s_k,\qquad b'_k=s_k,\]
\[c_k=s_k M_1\,\dist(x_k,X) +\frac{s_k\g_k \d_k}{2\a_{\min}}+8s_k^2\g_k^2.\]
Hence, the results follow from Lemma~\ref{lem-opt}.
\qed

\tat{Next we provide some sufficient conditions for the parameters to satisfy Assumption~\ref{assum-parameters} and formulate their properties which we will use in the forthcoming convergence rate analysis. 
\begin{lem}\label{lem:parameters}
Let us consider the following setting for the parameters: \[\mbox{$s_k = O\left(\frac{1}{k^c\ln^{(1+3g)/2}(k+1)}\right)$, $\gamma_k = O\left(\ln^g (k+1)\right)$, $\delta_k =  O\left(\frac{1}{k^d}\right)$ with $c,g,d>0$}.\] 
Then, if $c\in[1/2,1)$ and $d> 1/2$, Assumption~\ref{assum-parameters} holds. Moreover, given this choice of the parameters, 
$\sum_{k=1}^{t} s_k \ge \frac{(t+1)^{1-c}-1}{(1-c)\ln^{(1+3g)/2}(t+1)}$, 
$\sum_{k=1}^{t} s^2_k = O(1)$,  $\sum_{k=1}^{t} s_k\gamma_k\delta_k = O(1)$,  $\sum_{k=1}^{t} s^2_k\gamma^2_k = O(1)$ as $t\to\infty$. 
\end{lem}}

\begin{proof}
 Obviously, given the settings, $\lim_{k\to\infty}\gamma_k = \lim_{k\to\infty} \ln^g (k+1) = \infty$. 
 Moreover, 
 \begin{align*}
 	\sum_{k=1}^{t} s_k = \sum_{k=1}^{t} \frac{1}{k^c\ln^{(1+3g)/2}(k+1)} &\ge \frac{1}{\ln^{(1+3g)/2}(t+1)}\int_{1}^{t+1}\frac{dx}{x^c}\cr& = \frac{(t+1)^{1-c}-1}{(1-c)\ln^{(1+3g)/2}(t+1)}.
 \end{align*}
 On the other hand, 
  \begin{align*}
 	\sum_{k=1}^{t} s_k = \sum_{k=1}^{t} \frac{1}{k^c\ln^{(1+3g)/2}(k+1)} &\le \frac{1}{\ln^{(1+3g)/2}(2)}\left(1+\int_{1}^{t}\frac{dx}{x^c}\right)\cr& = \frac{t^{1-c}-c}{(1-c)\ln^{(1+3g)/2}(2)}.
 \end{align*}
Thus, $\sum_{k=1}^{t} s_k = O\left(\frac{t^{1-c}}{\ln^{(1+3g)/2}(2)}\right)$, which implies also the condition $\sum_{k=1}^{\infty} s_k = \infty$. Next, 
 \begin{align*}
	\sum_{k=1}^{t} s^2_k\gamma^2_k = \sum_{k=1}^{t} \frac{1}{k^{2c}\ln^{1+g}(k+1)} &\le \frac{1}{\ln^{1+g}2}+\int_{1}^{t}\frac{dx}{x\ln^{1+g}(x+1)} \cr&\le \frac{1}{\ln^{1+g}2}+\frac{g}{\ln^g 2}.
\end{align*}
Hence, $\sum_{k=1}^{t} s^2_k\le\sum_{k=1}^{t} s^2_k\gamma^2_k = O(1)$ and the corresponding series converge.  
Finally, 
 \begin{align*}
	\sum_{k=1}^{t} s_k\gamma_k\delta_k = \sum_{k=1}^{t} \frac{1}{k^{c+d}\ln^{(1+g)/2}(k+1)} &\le \frac{1}{\ln^{(1+g)/2}(2)}\left(1+\int_{1}^{t}\frac{dx}{x^{c+d}}\right)\cr& = \frac{t^{1-c-d}-c-d}{(1-c-d)\ln^{(1+3g)/2}(2)}.
\end{align*}
Thus, taking into account that $c+d>1$, we conclude that $\sum_{k=1}^{t} s_k\gamma_k\delta_k = O(1)$ and $\sum_{k=1}^{\infty} s_k\gamma_k\delta_k<\infty$.
\end{proof}


\subsection{Convergence Rate Results}
In this section, 
we establish convergence rate results for the iterates of the method~\eqref{eq:gradmet}, under 
the following condition:
there is a scalar $M>0$ such that
\begin{align}\label{eq:gradBound}
	\|\tilde\grad f(x_k)\|\le M,\qquad 
	\|\tilde\grad f(\Pi_X[x_k])\|\le M
	\qquad\hbox{for all  $k\ge1$}\ \ \as	
\end{align}
This condition is satisfied under Assumptions~\ref{asum:fconvex}--\ref{assum-parameters}. Specifically, as seen in Proposition~\ref{prop-convergence}, the norms of iterates are \as\ bounded by a constant.
As a consequence, the subgradients $\tilde\grad f(x_k)$ 
as well as the subgradients at the projection points $\tilde\grad f(\Pi_X[x_k])$ are also \as\ bounded by a constant.

In what follows, we allow for the strong convexity constant to take value 0 ($\m=0$) in order to provide a unified treatment of both cases when $f(\cdot)$ is strongly convex and when it is just convex. 

Using relation~\eqref{eq:gradBound}, we prove the following relations which will be useful in the further analysis.
\begin{lem}\label{lem-basic-iter-fvalues}
	Let $f(\cdot)$ be strongly convex with $\mu\ge 0$ and let the condition in~\eqref{eq:gradBound} hold.
	Then, the following relations are \as\ valid for the iterates $x_k$ of the method~\eqref{eq:gradmet} 
	for all $k\ge1$,
	\begin{itemize}
		\item[(a)]
		$f(\Pi_X[x_k])-f(x_k)\le M \dist(x_k,X)-\frac{\mu}{2}\dist^2(x_k,X),$
		\item[(b)] 
		$f(\Pi_X[x_k])-f(x_k)\ge -M \dist(x_k,X)+\frac{\mu}{2}\dist^2(x_k,X).$
	\end{itemize}
	Here, $\mu\ge0$ is the strong convexity constant for $f(\cdot)$ and $M>0$ is subgradient norm bound along the iterates and their projections on $X$ from~\eqref{eq:gradBound}.
\end{lem}
{\it Proof}:
Let $k\ge 1$ be arbitrary. 
By using the convexity of $f(\cdot)$ (see~\eqref{eq-strconvex}),
we estimate $f(\Pi_X[x_k])-f(x_k)$ as follows:
\begin{align*}
	f(\Pi_X[x_k])-f(x_k)
	&\le \la \tilde \nabla f(\Pi_X[x_k]), \Pi_X[x_k]- x_k\ra-\frac{\mu}{2}\|\Pi_X[x_k]-x_k\|^2\cr
	&\le M \|\Pi_X[x_k]-  x_k\|-\frac{\mu}{2}\|\Pi_X[x_k]-x_k\|^2,
\end{align*}
where the last inequality follows by the Cauchy-Schwarz inequality and~\eqref{eq:gradBound}. 
By  using $\|\Pi_X[x_k]-x_k\|=\dist(x_k,X)$ we obtain the relation in part (a).
The relation in part (b) is obtained similarly by using~\eqref{eq-strconvex}, where 
the roles of $x_k$ and $\Pi_X[x_k]$ are exchanged. \qed

The next lemma provides an iterate relation that is crucial in our convergence rate analysis of the method.
In the lemma, we use a scalar  $\eta\in[0,1]$ in order to estimate $f(y)-f(x_k)$ by a convex combination
of $f(y)-f(x_k)$ and $f(y)-f(\Pi_X[x_k])$ plus some additional terms. 
We will use $\eta=0$ or $\eta=1$,
depending on which quantities we want to estimate.
The lemma is a refinement of Lemma~\ref{lem-key},
which possible due to the condition~\eqref{eq:gradBound} and Lemma~\ref{lem-basic-iter}.

\begin{lem}\label{lem-key-improved}
	Let $f(\cdot)$ be strongly convex with $\mu\ge 0$ and let $\g_k>0$, $\d_k>0$, and $s_k>0$ for all $k$. Also, let the condition in~\eqref{eq:gradBound} hold. Then, the iterate sequence $\{x_k\}$ of the method~\eqref{eq:gradmet} \as\ satisfies
	for all $y\in X$ and $k\ge1$,
	\begin{align*}
		\ex{\|x_{k+1}- y\|^2\mid \EuScript F_k} 
		&\le(1-\mu s_k)\|x_k - y\|^2 
		+2s_k(1-\eta)(f(y) - f(x_k))\cr
		&\quad+2s_k\eta (f(y) - f(\Pi_X[x_k]))\cr
		&\quad +\frac{s_k\g_k \d_k}{2\a_{\min}} 
		- 2s_k\left(\frac{\g_k}{m\b} -\eta M\right)\dist(x_k,X)\cr
		&\quad - \eta\mu s_k\dist^2(x_k,X)
		+2s_k^2(M^2+\g_k^2),
	\end{align*}
	where $\eta\in[0,1]$ is arbitrary,  
	$\beta>0$ is the Hoffman constant from Lemma~\ref{lem-hoffman}, and
	$M>0$ is subgradient norm bound from~\eqref{eq:gradBound}.
\end{lem}

{\it Proof}:
By Lemma~\ref{lem-basic-iter}, we surely have for all
$y\in X$ and all $k\ge1$,
\begin{align*}
	&\|x_{k+1}- y\|^2 \le(1-\mu s_k)\|x_k - y\|^2 
	+2s_k(f(y)-f(x_k))\cr
	&+\frac{s_k\g_k \d_k}{2\a_{\min}} - 2s_k\g_k\dist(x_k,X_{i_k})
	+s_k^2(\|\tilde\nabla f(x_k)\| + \g_k)^2.
\end{align*}
By the condition in~\eqref{eq:gradBound}
the subgradients of $f(x_k)$ are bounded \as, so that $\|\tilde\nabla f(x_k)\| \le M$ \as\
Hence, we have \as\ for all $y\in X$ and all $k\ge1$,
\begin{align}\label{eq-rel-perk1}
	&\|x_{k+1}- y\|^2 \le(1-\mu s_k)\|x_k - y\|^2 
	+2s_k(f(y)- f(x_k)) \cr
	& +\frac{s_k\g_k \d_k}{2\a_{\min}}- 2s_k\g_k\dist(x_k,X_{i_k})
	+2s_k^2(M^2+\g_k^2).\qquad
\end{align}
Taking the conditional expectation with respect to $\EuScript F_k$ in relation~\eqref{eq-rel-perk1}, 
we obtain  that \as\, for all $y\in X$ and all $k\ge1$,
\begin{align*}
	&\ex{\|x_{k+1}- y\|^2\mid \EuScript F_k} \le
	(1-\mu s_k) \|x_k - y\|^2 +2s_k(f(y) - f(x_k))\cr
	& \ \ +\frac{s_k\g_k \d_k}{2\a_{\min}} - 2s_k\g_k \ex{\dist(x_k,X_{i_k})\mid \EuScript F_k} 
	+2s_k^2(M^2+\g_k^2).
\end{align*} 
Given $x_k$, the index $i_k$ is uniformly distributed over $\{1,\ldots,m\}$,  so it follows that
$\ex{\dist(x_k,X_{i_k})\mid \EuScript F_k} =\frac{1}{m}\sum_{i=1}^m\dist(x_k,X_i),$
and by using Hoffman's Lemma (Lemma~\ref{lem-hoffman}), we obtain
\[\ex{\dist(x_k,X_{i_k})\mid \EuScript F_k}\ge \frac{1}{\b m}\dist(x_k,X).\]
Therefore, we have \as\, for all $y\in X$ and all $k\ge1$,
\begin{align}\label{eq-prima}
	& \ex{\|x_{k+1}- y\|^2\mid \EuScript F_k} \le
	(1-\mu s_k) \|x_k - y\|^2 +2s_k(f(y) - f(x_k))\cr
	& \ \ +\frac{s_k\g_k \d_k}{2\a_{\min}}
	- \frac{2s_k\g_k}{m\b} \dist(x_k,X)+2s_k^2(M^2+\g_k^2).
\end{align}

Let $\eta\in[0,1]$ and lets write
\[f(y) - f(x_k)=(1-\eta)(f(y) - f(x_k))+\eta(f(y) - f(x_k)).\]
Further, we write
\[\eta(f(y) - f(x_k))=\eta(f(y) - f(\Pi_X[x_k] +f(\Pi_X[x_k]) - f(x_k)).\]
By Lemma~\ref{lem-basic-iter-fvalues}(a), we have
\[f(\Pi_X[x_k]) - f(x_k)\le M \dist(x_k,X)-\frac{\mu}{2}\dist^2(x_k,X).\]
Combining the preceding relations yields
\begin{align*}
	f(y) - f(x_k)&\le(1-\eta)(f(y) - f(x_k)) + \eta(f(y) - f(\Pi_X[x_k]) \cr
	&+ \eta M \dist(x_k,X)- \frac{\eta\mu}{2}\dist^2(x_k,X).
\end{align*}
By substituting the preceding estimate back in~\eqref{eq-prima}, we obtain the stated relation.
\qed

\subsubsection{Convergence rate in merely convex case}
\def\xtav{{x_t^{\textrm av}}}
\def\xttav{{x_{\tau,t}^{\textrm av}}}
\def\xtktav{{x_{k_1,t}^{\textrm av}}}
Here, we provide convergence rate result for the method in the case when $f(\cdot)$ is merely convex, i.e., 
$\mu=0$ in relation~\eqref{eq-strconvex}. 
To obtain a convergence rate estimate for the method, we will 
consider the weighted averages and the truncated weighted averages of the iterates. We define the weighted averages with respect to a generic sequence $\{\nu_k\}$  of positive scalars.
Given a positive sequence $\{\nu_k\}$, we define the $\nu$-weighted averages, as follows:
\begin{align}\label{eq-waver}
	\xtav &=S_t^{-1}\sum_{k=1}^t \nu_k x_k,\quad S_t =\sum_{k=1}^t \nu_k,\qquad\hbox{for all }t\ge 1.
\end{align}
Let us focus on $\xtav$ defined in~\eqref{eq-waver} as follows.
First, let us note that the truncated $\nu$-weighted average $\xttav$ of the iterates $x_k$ for $k=\tau,\ldots,t$, is defined by
\begin{align}\label{eq-truncwaver}
	x_{\tau, t}^{\textrm av}=(S_t-S_\tau)^{-1}\sum_{k=\tau}^t \nu_k x_k\qquad\hbox{for all }t\ge \tau\ge 1.
\end{align}
The $\nu$-weighted average $\xtav$ can be related to
the average $x_\tau^{\textrm av}$, with $\tau\le t$, and the truncated $\nu$-weighted average $\xttav$. Specifically, for any $t\ge \tau$, we have\vspace{-0.25cm}
\[\xtav=\frac{\sum_{k=1}^t \nu_k x_k}{\sum_{k=1}^t \nu_k}
=\frac{\sum_{k=1}^\tau \nu_k x_k +\sum_{k=\tau}^t \nu_k x_k}{\sum_{k=1}^t \nu_k}.\]
By using the definitions of $S_t$, $x_\tau^{\textrm av}$, and $\xttav$ i
n~\eqref{eq-waver}--\eqref{eq-truncwaver}, it follows that 
\be\label{eq-averages}
\xtav=\frac{S_\tau x_{\tau}^{\textrm av} 
	+ (S_t-S_\tau)\xttav}{S_t}\qquad\hbox{for all }t\ge\tau\ge 1.\ee
Thus, $\xtav$ is a convex combination of $x_{\tau}^{\textrm av}$ and
the truncated-weighted average $\xttav$. 

The following lemma provides some preliminary estimates regarding $\xtav$.
\begin{lem}\label{lem-prel-est}
	Let $f(\cdot)$ be convex, and 
	let the condition in~\eqref{eq:gradBound} hold. Assume that the problem $\min_{x\in X}f(x)$ has an optimal solution. Also,
	assume that $s_k>0$, $\g_k>0$, $\d_k>0$, and
	$\g_{k+1}\ge \g_k$ for all $k$, and $\lim_{k\to\infty}\g_k=\infty$.
	Let $\{x_k\}$ be generated by the method~\eqref{eq:gradmet}. Consider 
	the $s$-weighted averages $\{\xtav\}$ of the iterates obtained by using
	$\nu_k=s_k$ in~\eqref{eq-waver} for all $k\ge1$.
	Then, the following estimates are valid:
	for all $t\ge1$,
	\[\ex{f(\xtav)}-f^*\le \frac{\ex{\|x_1 - x^*\|^2}}{2S_t}
	+\frac{\sum_{k=1}^t\b_k}{2S_t}\ \hbox{for  all $x^*\in X^*$},\] 
	where $\b_k=\frac{s_k\g_k \d_k}{2\a_{\min}}+2s_k^2(M^2+\g_k^2)$
	for all $k$, with $M$ being the constant from condition~\eqref{eq:gradBound}.
	Moreover, there exists $k_1\ge 1$ such that for all $t\ge k_1$,
	\begin{align*}
		\ex{\dist(\xtav,X)}&\le 
		\frac{S_{k_1}}{S_t}\ex{\dist(x_{k_1}^{\textrm av},X)}
		+\frac{m\b\g_{k_1}^{-1}}{S_t}\ex{\dist^2(x_{k_1},X)}
		+\frac{m\b\sum_{k=k_1}^t \g_k^{-1}\b_k}{S_t}.\end{align*}
	
\end{lem}

{\it Proof}:
In Lemma~\ref{lem-key-improved}, 
we let $\mu=0$ and $\eta=0$, and thus obtain
\as\, for all $y\in X$ and $k\ge1$,
\begin{align*}
	&\ex{\|x_{k+1}- y\|^2\mid \EuScript F_k} \le\|x_k - y\|^2 +2s_k(f(y) - f(x_k))\cr
	& +\frac{s_k\g_k \d_k}{2\a_{\min}}- \frac{2s_k\g_k}{m\b} \dist(x_k,X)+2s_k^2(M^2+\g_k^2).
\end{align*}
By taking the total expectation
in the preceding relation, we obtain for all $y\in X$ and $k\ge1$,
\begin{align}\label{eq-777}
	&\ex{\|x_{k+1}- y\|^2} \le \ex{\|x_k - y\|^2} +2s_k\ex{f(y) - f(x_k)}
	- \frac{2s_k\g_k}{m\b} \ex{\dist(x_k,X)}+\b_k,\qquad
\end{align}
where 
$\b_k=\frac{s_k\g_k \d_k}{2\a_{\min}}+2s_k^2(M^2+\g_k^2)$.
By summing these relations over $k=1,\ldots, t$ and 
by dropping the distance-related terms, after re-arranging the remaining terms, 
we obtain for all $y\in X$ and all $t\ge 1,$
\begin{align*}
	&2 \sum_{k=1}^t s_k\left(\ex{f(x_k)}- f(y)\right)
	+\ex{\|x_{t+1}- y\|^2} \le\ex{\|x_1 - y\|^2}+\sum_{k=1}^t\b_k.
\end{align*}
Next, we divide the preceding relation with $2 S_t$, where $S_t=\sum_{k=1}^t s_k$, and we use the definition of $\xtav$  (i.e.,~\eqref{eq-waver} with $\nu_k=s_k$).
This and the convexity of $f(\cdot)$ yield for $y=x^*$, where $x^*\in X^*$ is arbitrary, 
and for all $t\ge1$,
\[\ex{f(\xtav)}-f^*\le \frac{\ex{\|x_1 - x^*\|^2}}{2S_t}
+\frac{\sum_{k=1}^t\b_k}{2S_t}.\]

To show the estimate for $\ex{\dist(\xtav,X)}$, we use 
relation~\eqref{eq-777} with $y=\Pi_X[x_k]\in X$, and
obtain for all $k\ge1$,
\begin{align*}
	&\ex{\|x_{k+1}- \Pi_X[x_k]\|^2} \le\ex{\|x_k - \Pi_X[x_k]\|^2} \cr
	&+2s_k\ex{f(\Pi_X[x_k]) - f(x_k)}
	- \frac{2s_k\g_k}{m\b} \ex{\dist(x_k,X)}+\b_k.
\end{align*}
Since $\ex{\dist(x_{k+1},X)}\le \ex{\|x_{k+1}- \Pi_X[x_k]\|}$
and $\|x_k - \Pi_X[x_k]\|=\dist(x_k,X)$ it follows that 
for all $k\ge1$,
\begin{align*}
	\ex{\dist^2(x_{k+1},X)} \le&\ex{\dist^2(x_k,X)} +2s_k\ex{f(\Pi_X[x_k]) - f(x_k)}\cr
	&- \frac{2s_k\g_k}{m\b}\ex{\dist(x_k,X)}+\b_k.
\end{align*}
By Lemma~\ref{lem-basic-iter-fvalues}(a), where $\mu=0$,
we have that 
\[f(\Pi_X[x_k]) - f(x_k)\le M\dist(x_k,X),\]
implying that  for all $k\ge1$,
\begin{align*}
	\ex{\dist^2(x_{k+1},X)} \le\ex{\dist^2(x_k,X)} -2s_k\left(\frac{\g_k}{m\b}-M\right)\ex{\dist(x_k,X)}+\b_k.
\end{align*}

Since $\g_k\to+\infty$, there exists a large enough $k_1$ such that
$\g_k/(m\b)-M\ge \g_k/(2m\b)$, implying that for all $k\ge k_1$,
\begin{align*}
	\ex{\dist^2(x_{k+1},X)} \le\ex{\dist^2(x_k,X)} -\frac{s_k\g_k}{m\b}\ex{\dist(x_k,X)}+\b_k.
\end{align*}
Dividing the preceding relation with $\g_k$ and using the assumption that 
$\g_{k+1}\ge \g_k$, we obtain for all $k\ge k_1$,
\begin{align*}
	\g_{k+1}^{-1}\ex{\dist^2(x_{k+1},X)} &\le\g_k^{-1}\ex{\dist^2(x_k,X)} -\frac{s_k}{m\b}\ex{\dist(x_k,X)}+\g_k^{-1}\b_k.
\end{align*}
By summing these relations over $k=k_1,\ldots,t,$
after re-arranging the terms, we can see that for all $t\ge k_1$, we have
\begin{align*}
	\sum_{k=k_1}^t \frac{s_k}{m\b}\ex{\dist(x_k,X)} &\le\g_{k_1}^{-1}\ex{\dist^2(x_{k_1},X)}
	+\sum_{k=k_1}^t\frac{\b_k}{\g_k}.
\end{align*}
Since $X$ is convex, the distance function $\dist(\cdot,X)$ is convex, and 
by dividing the preceding relation with $\sum_{k=k_1}^t s_k$, we obtain 
for all $t\ge k_1$,
\begin{align}\label{eq-88}
	\frac{1}{m\b}\ex{\dist(\xtktav,X)} 
	\le\frac{\g_{k_1}^{-1}\ex{\dist^2(x_{k_1},X)}}{S_t-S_{k_1}}
	+\frac{\sum_{k=k_1}^t \g_k^{-1}\b_k}
	{S_t-S_{k_1}},
\end{align}
where we used the notation $S_t=\sum_{k=1}^t s_k$.

By expressing $\xtav$ as a convex combination of $x_{k_1}^{\textrm av}$ and $\xtktav$, i.e., by using~\eqref{eq-averages} with $\nu_k=s_k$ and $\tau=k_1$, 
\[\xtav=\frac{S_\tau x_{\tau}^{\textrm av} 
	+ (S_t-S_\tau)\xttav}{S_t}\qquad\hbox{for all }t\ge k_1\ge 1.\]
Then, by using the convexity of the distance function $\dist(\cdot,X)$,
we find that for all $t\ge k_1$,
\begin{align*}
	\ex{\dist(\xtav,X)}
	\le 
	\frac{S_{k_1}}{S_t}\ex{\dist(x_{k_1}^{\textrm av},X)} \ +\frac{S_t-S_{k_1}}{S_t}\ex{\dist(\xtktav,X)}.\end{align*}
The preceding relation and~\eqref{eq-88} yield for all $t\ge k_1$,
\begin{align*}
	\ex{\dist(\xtav,X)}\le& 
	\frac{S_{k_1}}{S_t}\ex{\dist(x_{k_1}^{\textrm av},X)}\cr &+\frac{m\b\g_{k_1}^{-1}}{S_t}\ex{\dist^2(x_{k_1},X)}
	+\frac{m\b\sum_{k=k_1}^t \g_k^{-1}\b_k}{S_t}.\end{align*}
\qed

%


The following proposition provides the convergence rate for the expected function values estimated at $\xtav$. 
\tat{\begin{prop}\label{prop-converg_rate}
	Let $f(\cdot)$ be convex, and 
	let the condition in~\eqref{eq:gradBound} hold. Assume that the problem $\min_{x\in X}f(x)$ has an optimal solution. Also, assume that \[\mbox{$s_k = O\left(\frac{1}{k^c\ln^{(1+3g)/2}(k+1)}\right)$, $\gamma_k = O\left(\ln^g (k+1)\right)$, $\delta_k =  O\left(\frac{1}{k^d}\right)$ }\]
	with $g>0$, $c\in[1/2,1)$ and $d> 1/2$.
	Let $\{x_k\}$ be generated by the method~\eqref{eq:gradmet}. Consider 
	the $s$-weighted averages $\{\xtav\}$ of the iterates obtained by using	
	$\nu_k=s_k$ in~\eqref{eq-waver} for all $k\ge1$.
	Then, the following estimate is valid, \an{as $t\to\infty$}, \vspace{-0.2cm}
	\[|\ex{f(\xtav)}-f^*| = O\left(\frac{\ln^{1/2+3g/2} t}{t^{1-c}}\right).\] 
\end{prop}}

{\it Proof}:
Let us notice that the setting for the parameters $s_k$, $\gamma_k$, and $\delta_k$ above is the same as one in Lemma~\ref{lem:parameters}.
Thus, by combining the result provided by that lemma with \an{the result for function values} 
from Lemma~\ref{lem-prel-est}, namely,
\[\ex{f(\xtav)}-f^*\le \frac{\ex{\|x_1 - x^*\|^2}}{2S_t}
+\frac{\sum_{k=1}^t\b_k}{2S_t}\quad \hbox{for  all $x^*\in X^*$},\] 
where $\b_k=\frac{s_k\g_k \d_k}{2\a_{\min}}+2s_k^2(M^2+\g_k^2)$, we conclude that there exists a scalar $C_1>0$ such that for sufficiently large $t$
\begin{align}\label{eq:upper}
	\ex{f(\xtav)}-f^*\le \frac{C_1 \ln^{1/2+3g/2} t}{t^{1-c}}.
\end{align}  
The constant $C_1$ depends on $M^2$ (where $M$ is from the condition~\eqref{eq:gradBound}), the smallest norm $\a_{\min}$ of the vectors $a_i,i\in[m]$, and the squared distance between the initial point and the corresponding solution $\ex{\|x_1 - x^*\|^2}$.

On the other hand, using the result on expected distances of Lemma~\ref{lem-prel-est} and the inequality $\ex{f(\xtav)}-f^*
\ge -\|\tilde \nabla f(x^*)\|\,\ex{\dist(\xtav,X)}$ (see~\eqref{eq:lowerbound}), we conclude that there exists finite $k_1$ such that for all $t>k_1$,
\begin{align*}&\ex{f(\xtav)}-f^*
	\ge -\|\tilde \nabla f(x^*)\|\,\ex{\dist(\xtav,X)} \cr
	& \ge -\|\tilde \nabla f(x^*)\|\left(\frac{S_{k_1}}{S_t}\ex{\dist(x_{k_1}^{\textrm av},X)}+\frac{m\b\g_{k_1}^{-1}}{S_t}\ex{\dist^2(x_{k_1},X)}
	+\frac{m\b\sum_{k=k_1}^t \g_k^{-1}\b_k}{S_t}\right),
\end{align*} 
where $\tilde \nabla f(x^*)$ is a subgradient of $f(\cdot)$ at some optimal point $x^*\in X^*$.
By using Lemma~\ref{lem:parameters} again, we conclude the existence of some constant $C_2>0$ such that
\begin{align}\label{eq:lower}
	\ex{f(\xtav)}-&f^*\ge -\frac{C_2\ln^{1/2+3g/2} t}{t^{1-c}}.
\end{align} 
The constant $C_2$ in its turn depends on the index $k_1$, the upper bound $M$ of the subgradient norms, the number of the constraints $m$, the value $\a_{\min}$, Hoffman's constant $\beta$, and distance between $x_{k_1}$ and the feasible set $X$.
Combining~\eqref{eq:upper} and~\eqref{eq:lower}, we obtain the stated result.
\qed

{\it By optimizing the parameters $c$ and $d$ in Proposition~\ref{prop-converg_rate},
	the convergence rate of the order $O\left(\frac{\ln^{1/2+\epsilon} t}{t^{1/2}}\right)$ for any small positive $\varepsilon$ is obtained for $c=1/2$ and $g = 2\epsilon/3$.}
The constant involved in $O$-notation in Proposition~\ref{prop-converg_rate} can be explicitly derived using Lemma~\ref{lem-prel-est}, and it \emph{linearly} depends on $M^2$ (with $M$ from~\eqref{eq:gradBound}), the Hoffman constant $\beta$, the number of the constraints $m$, the smallest inverse norm $\a^{-1}_{\min}$ of the vectors $a_i,i\in[m]$, and the expected squared distance from the initial point $x_1$ to the solution set $X^*$.

\def\bxtav{{\bar{x}_t^{\textrm av}}}
\def\bxttav{{\bar{x}_{\tau,t}^{\textrm av}}}
\def\bxtktav{{\bar{x}_{k_1,t}^{\textrm av}}}
\subsubsection{Convergence rate in strongly convex case}
Here, we analyze the convergence rate of the method~\eqref{eq:gradmet} for strongly convex
$f(\cdot)$ with $\mu>0$ (see~\eqref{eq-strconvex}). We note that in this case the original
problem of minimizing $f(x)$ over $x\in X$ has a unique solution (as $X$ is assumed to be nonempty).
To establish the convergence rate of the method, we will consider the $s^{-1}$-
weighted averages $\xtav$ defined by~\eqref{eq-waver} with $\nu_k=s_k^{-1}$.
To differentiate these averages from those used in the preceding section, we define\vspace{-0.25cm}
\begin{align}\label{eq-waver-new}
	\bxtav &=\bar {S}_t^{-1}\sum_{k=1}^t s_k^{-1} x_k\quad, \bar{S}_t=\sum_{k=1}^t s_k^{-1},\qquad\hbox{for all }t\ge 1.
\end{align}
\begin{rem}
	The $s^{-1}$-weighted averages of the form~\eqref{eq-waver-new} have been proposed in~\cite{ned2014} to analyze
	the convergence rate of a stochastic gradient method in terms of the expected function values.
	As seen in~\cite{ned2014}, using the $s^{-1}$-weighted averages for a stochastic gradient method applied to 
	minimizing merely convex function $f(\cdot)$ yields
	the convergence rate of $O(1/\sqrt{k})$, where $k$ is the number of iterations.
	However, in this case, an additional assumption that the iterates are bounded is needed, which is the main reason why we did not consider the $s^{-1}$-weighted averages in the preceding section.
\end{rem}
\begin{rem}
	The $s^{-1}$-weighted averages of the form~\eqref{eq-waver-new} have the convergence rate of $O(1/k)$
	for a stochastic subgradient method as applied to a strongly convex function $f(\cdot)$, as shown in~\cite{ned2014}. Unlike the setting in~\cite{ned2014}, here in addition to time-varying functions,  we have to deal with the infeasibility of the iterates. 
\end{rem}

Using the weighted average $\bxtav$ given in~\eqref{eq-waver-new}, we define 
the truncated weighted average $\bxttav$ for $k=\tau,\ldots,t$, as follows
\begin{align}\label{eq-truncwaver-new}
	\bxttav=(S_t-S_\tau)^{-1}\sum_{k=\tau}^t \nu_k x_k\qquad\hbox{for all }t\ge \tau\ge 1.
\end{align} with $\nu_k=s_k^{-1}$ for all $k$.
Note that the $s^{-1}$-weighted average $\bxtav$ and the truncated
weighted average $\bar{x}_\tau^{\textrm av}$, with $\tau\le t$, satisfy the relation
\be\label{eq-averages-new}
\bxtav=\frac{\bar{S}_\tau \bar{x}_{\tau}^{\textrm av} 
	+ (\bar{S}_t-\bar{S}_\tau)\bxttav}{\bar{S}_t}\qquad\hbox{for all }t\ge \tau\ge 1.\ee
A preliminary result regarding the function value at $\bxtav$  is given in the following lemma. 

\begin{lem}\label{lem-prel-est-strong}
	Let $f(\cdot)$ be strongly convex with $\mu>0$, and let the condition in~\eqref{eq:gradBound}) hold.
Assume that $s_k=\frac{2}{\mu k}$. 
Also, assume that $\g_k>0$, and 
	$\d_k>0$.
		Let $\{x_k\}$ be the iterate sequence  generated by method~\eqref{eq:gradmet}.
	Then, the following estimates are valid for the $s^{-1}$-weighted averages $\bxtav$:
	\an{for all $t\ge1$},
	\[\ex{f(\bxtav)}-f^*\le \frac{s_0^{-2}\ex{\|x_1 - x^*\|^2}}{2\bar{S}_t}
	+\frac{\sum_{k=1}^t c_k}{2\bar{S}_t},\]
	\an{where $s_0=2\mu^{-1}$, 
		$c_k=\frac{\g_k \d_ks_k^{-1}}{2\a_{\min}}+2(M^2+\g_k^2)$
		for all $k$, with $M$ being the constant from~\eqref{eq:gradBound}.}	
	Moreover, there exists $k_1\ge 1$ such that for all $t\ge k_1$,
	\begin{align*}
		\ex{\dist(\bxtav,X)}\le 
		\frac{\bar{S}_{k_1}}{\bar{S}_t}\ex{\dist(\bar{x}_{k_1}^{\textrm av},X)} 
		&+\frac{m\b\g_{k_1}^{-1}s_{k_1-1}^{-2}}{\bar{S}_t}\ex{\dist^2(x_{k_1},X)}\cr
		&+\frac{m\b\sum_{k=k_1}^t \g_k^{-1}c_k}{\bar{S}_t}.\end{align*}
\end{lem}

\begin{proof}
	In Lemma~\ref{lem-key-improved}, where $\mu>0$, we let $\eta=0$, and thus obtain
	\as\ for all $y\in X$ and $k\ge1$,
	\begin{align*}
		\ex{\|x_{k+1}- y\|^2\mid \EuScript F_k} 
		&\le(1-\mu s_k)\|x_k - y\|^2 
		+2s_k(f(y) - f(x_k))+\frac{s_k\g_k \d_k}{2\a_{\min}}\cr
		& \quad- \frac{2s_k\g_k}{m\b} \dist(x_k,X)+2s_k^2(M^2+\g_k^2).
	\end{align*}
	By taking the total expectation and dividing with $s_k^2$, we obtain for all $y\in X$ and $k\ge1$,
	\begin{align*}
		&s_k^{-2}\ex{\|x_{k+1}- y\|^2} 
		\le s_k^{-2}(1-\mu s_k)\ex{\|x_k - y\|^2} \cr 
		&+2s_k^{-1}\ex{f(y) - f(x_k)}
		- \frac{2s_k^{-1}\g_k}{m\b} \ex{\dist(x_k,X)}
		+c_k,\qquad
	\end{align*}
	where
	\[c_k=\frac{\g_k \d_ks_k^{-1}}{2\a_{\min}}+2(M^2+\g_k^2).\]
	Note that for the given stepsize $s_k=\frac{2}{\mu k}$,
	we have for sufficiently large $k$
	\[ s_k^{-2}(1-\mu s_k)=\frac{\mu^2k(k-2)}{4}
	\le \frac{\mu^2 (k-1)^{2}}{4}=s_{k-1}^2.\]
	Therefore, we obtain for all $y\in X$ and sufficiently large $k\ge1$,
	\begin{align}\label{eq-77ag}
		&s_k^{-2}\ex{\|x_{k+1}- y\|^2} 
		\le s_{k-1}^{-2}\ex{\|x_k - y\|^2} \cr 
		&+2s_k^{-1}\ex{f(y) - f(x_k)}
		- \frac{2s_k^{-1}\g_k}{m\b} \ex{\dist(x_k,X)}
		+c_k.\qquad\qquad
	\end{align}
	By summing the relations in~\eqref{eq-77ag} over $k=1,\ldots, t$ and omitting the terms with $\ex{\dist(x_k,X)}$,
	after re-arranging the terms, 
	we find that for all $y\in X$ and all $t\ge 1,$
	\begin{align*}
		2 \sum_{k=1}^t s_k^{-1}\left(\ex{f(x_k)}- f(y)\right)
		+s_t^{-2}\ex{\|x_{t+1}- y\|^2} 
		\le s_0^{-2}\ex{\|x_1 - y\|^2}+\sum_{k=1}^t c_k,
	\end{align*}
	where we define $s_0=2\mu^{-1}$.
	By dividing with $2\bar S_t$, where $\bar{S}_t=\sum_{k=1}^t s_k^{-1}$, using the definition of
	$\bxtav$ (see~\eqref{eq-waver-new}) and the convexity of $f(\cdot)$, 
	we obtain for $y=x^*$ and all $t\ge1$,
	\[\ex{f(\bxtav)}-f^*\le \frac{s_0^{-2}\ex{\|x_1 - x^*\|^2}}{2\bar{S}_t}
	+\frac{\sum_{k=1}^t c_k}{2\bar{S}_t}.\]
	
	Next, we establish the estimate for $\ex{\dist(\bxtav,X)}$. We use 
	relation~\eqref{eq-77ag} with $y=\Pi_X[x_k]\in X$. Noting that 
	$\ex{\dist(x_{k+1},X)}\le \ex{\|x_{k+1}- \Pi_x[x_k]\|^2}$
	and $\|x_k - \Pi_X[x_k]\|=\dist(x_k,X)$ for all $k$, we obtain for all $k\ge1$,
	\begin{align*}
		s_k^{-2}\ex{\dist^2(x_{k+1},X)} 
		&\le s_{k-1}^{-2}\ex{\dist^2(x_k,X)} +2s_k^{-1}\ex{f( \Pi_X[x_k]) - f(x_k)}\cr
		&\quad- \frac{2s_k^{-1}\g_k}{m\b} \ex{\dist(x_k,X)}
		+c_k.\qquad\qquad
	\end{align*}
	\an{By the convexity of $f(\cdot)$ and 
		the definition of the subgradient bound $M$ in~\eqref{eq:gradBound},}
	we have that 
	\[f(\Pi_X[x_k]) - f(x_k)\le M\dist(x_k,X)\]
	(see Lemma~\ref{lem-basic-iter-fvalues}(a) where the term $\dist^2(x_k,X)$ is omitted). Therefore, for all $k\ge1$,
	\begin{align*}
		s_k^{-2}\ex{\dist^2(x_{k+1},X)} 
		\le s_{k-1}^{-2}\ex{\dist^2(x_k,X)}  
		&- 2s_k^{-1} \left(\frac{\g_k}{m\b}-M\right)\ex{\dist(x_k,X)}\cr
		&+c_k.\qquad\qquad
	\end{align*}
	As $\g_k\to+\infty$, there is a large enough $k_1$ so that
	$\g_k/(m\b)-M\ge \g_k/(2m\b)$, implying that for all $k\ge k_1$,
	\begin{align*}
		s_k^{-2}\ex{\dist^2(x_{k+1},X)} 
		\le s_{k-1}^{-2}\ex{\dist^2(x_k,X)} -&\frac{s_k^{-1}\g_k}{m\b}\ex{\dist(x_k,X)}\cr
		&+c_k.
	\end{align*}
	Now, we divide the preceding relation with $\g_k$. By using the assumption that 
	$\g_{k+1}\ge \g_k$, we obtain for all $k\ge k_1$,
	\begin{align*}
		\g_{k+1}^{-1}s_k^{-2} \ex{\dist^2(x_{k+1},X)} \le\g_k^{-1}s_{k-1}^{-2}\ex{\dist^2(x_k,X)} 
		&-\frac{s_k^{-1}}{m\b}\ex{\dist(x_k,X)}\cr
		&+\g_k^{-1}c_k.
	\end{align*}
	By summing these relations over $k=k_1,\ldots,t,$
	after re-arranging the terms, we can see that for all $t\ge k_1$,
	\begin{align*}
		\sum_{k=k_1}^t \frac{s_k^{-1}}{m\b}\ex{\dist(x_k,X)} &\le\g_{k_1}^{-1}s_{k_1-1}^{-2}\ex{\dist^2(x_{k_1},X)}
		+\sum_{k=k_1}^t\frac{c_k}{\g_k}.
	\end{align*}
	The distance function $\dist(\cdot,X)$ is convex since $X$ is convex. Hence, 
	upon dividing the preceding relation with $\sum_{k=k_1}^t s_k^{-1}$, and by using the convexity of $\dist(\cdot,X)$
	and the definition of the truncated $s^{-1}$-weighted averages (see~\eqref{eq-waver-new}, and~\eqref{eq-truncwaver} with $\nu_k=s_k^{-1}$), we obtain 
	for all $t\ge k_1$,
	\begin{align}\label{eq-88ag}
		\frac{1}{m\b}\ex{\dist(\bxtktav,X)} 
		\le\frac{\g_{k_1}^{-1}s_{k_1-1}^{-2}\ex{\dist^2(x_{k_1},X)}}{\bar{S}_t-\bar{S}_{k_1}}
		+\frac{\sum_{k=k_1}^t \g_k^{-1}c_k}
		{\bar{S}_t-\bar{S}_{k_1}},
	\end{align}
	where we used the notation $\bar{S}_t=\sum_{k=1}^t s_k^{-1}$ (see~\eqref{eq-waver-new}).
	
	By expressing $\bxtav$ as a convex combination of $\bar{x}_{k_1}^{\textrm av}$ and $\bxtktav$, i.e., by using~\eqref{eq-averages-new} with $\nu_k=s_k^{-1}$ and $\tau=k_1$, we have
	\[\bxtav=\frac{\bar{S}_{k_1} \bar{x}_{{k_1}}^{\textrm av} 
		+ (\bar{S}_t-\bar{S}_{k_1})\bar{x}_{{k_1},t}^{\textrm av} }{\bar{S}_t}\qquad\hbox{for all }t\ge k_1\ge 1.\]
	Then, by using the convexity of the distance function $\dist(\cdot,X)$,
	we find that for all $t\ge k_1$,
	\begin{align*}
		\ex{\dist(\bxtav,X)}\le 
		\frac{\bar{S}_{k_1}}{\bar{S}_t}\ex{\dist(\bar{x}_{k_1}^{\textrm av},X)} 
		+\frac{\bar{S}_t - \bar{S}_{k_1}}{\bar{S}_t}\ex{\dist(\bxtktav,X)}.\end{align*}
	The preceding relation and~\eqref{eq-88ag} yield for all $t\ge k_1$,
	\begin{align*}
		\ex{\dist(\bxtav,X)}\le 
		\frac{\bar{S}_{k_1}}{\bar{S}_t}\ex{\dist(\bar{x}_{k_1}^{\textrm av},X)} 
		&+\frac{m\b\g_{k_1}^{-1}s_{k_1-1}^{-2}}{\bar{S}_t}\ex{\dist^2(x_{k_1},X)}\cr
		&+\frac{m\b\sum_{k=k_1}^t \g_k^{-1}c_k}{\bar{S}_t}.\end{align*}
\end{proof}

The next proposition provides the convergence rate of the procedure~\eqref{eq:gradmet} in terms of the $s^{-1}$-
weighted averages $\bxtav$ for a strongly convex function $f$.
\tat{\begin{prop}\label{prop-converg_rate-strong}
	Let $f(\cdot)$ be strongly convex with $\mu>0$, and 
	let Assumption~\ref{assum-infty} hold.
	Assume $s_k=\frac{2}{\mu k}$, $\d_k = \frac{1}{k^d}$ \an{with} $d>1$, and
	$\g_{k} = \ln^g k$ with $g>0$ for all $k$. 
	Let $\{x_k\}$ be the iterate sequence  generated by method~\eqref{eq:gradmet}.
	Then, the following estimate is valid for the $s^{-1}$-weighted averages $\bxtav$, $t\to\infty$, \vspace{-0.2cm}
	\[|\ex{f(\bxtav)}-f^*| = O\left(\frac{\ln^{2g} t}{t}\right). \]
\end{prop}}
\begin{proof}
	Given the choice $s_k =\frac{2}{\mu k}$, 
	\begin{align}\label{eq:barS_t}
		\bar{S}_t = \sum_{k=1}^{t}\frac{\mu k}{2} = \frac{\mu}{2} \frac{(1+t)t}{2}.
	\end{align}
	Next, taking into account that $\g_k = \ln^g k$ and $\d_k = \frac{1}{k^d}$, $d>1$, we obtain 
	\begin{align}\label{eq:c_t}
		\sum_{k=1}^t s^{-1}_k\g_k\d_k &=\sum_{k=1}^t \frac{\mu\ln^g k}{2k^{d-1}} \cr
		&\le \frac{\mu\ln^g t}{2} [\int_{1}^t\frac{dx}{x^{d-1}} + 1] =\begin{cases}
									  \frac{\mu\ln^g t}{2}\left(\frac{t^{2-d}-1}{2-d}+1\right), \, \mbox{if $d\ne 2$},\\
									  \frac{\mu\ln^g t}{2}\left(\ln t+1\right), \, \mbox{if $d= 2$}.
		                                                              \end{cases}
	\end{align} 
	By combining the inequalities in~\eqref{eq:barS_t} and~\eqref{eq:c_t} with \an{the estimate for the function values 
		from} Lemma~\ref{lem-prel-est-strong}, namely,
	\[\ex{f(\bxtav)}-f^*\le \frac{s_0^{-2}\ex{\|x_1 - x^*\|^2}}{2\bar{S}_t}
	+\frac{\sum_{k=1}^t c_k}{2\bar{S}_t},\]
	where $s_0=2\mu^{-1}$, 
	$c_k=\frac{\g_k \d_ks_k^{-1}}{2\a_{\min}}+2(M^2+\g_k^2)$, we conclude that there exists the constant $D_1>0$ such that
	\begin{align}\label{eq:upper1}
		\ex{f(\bxtav)}-f^*\le \frac{D_1 \ln^{2g} t}{t}.
	\end{align}  
	Note that $D_1$ depends on \an{the squared upper bound $M^2$ of the subgradients defined in relation~\eqref{eq:gradBound}, the value $\a_{\min}$, and the expected} squared distance between the initial point and the corresponding solution $\ex{\|x_1 - x^*\|^2}$.
	
	\an{To determine a lower bound for $\ex{f(\bxtav)}-f^*$, we use the estimate for the expected distance from Lemma~\ref{lem-prel-est-strong} and $\ex{f(\bxtav)}-f^*
\ge -\|\tilde \nabla f(x^*)\|\,\ex{\dist(\bxtav,X)}$ (see~\eqref{eq:lowerbound}), taking into account that the solution set $X^*$ is singleton, i.e., $X^*=\{x^*\}$. Without loss of generality, we may assume that the subgradient $\tilde\nabla f(x^*)$  satisfies 
		$\|\tilde\nabla f(x^*)\|\le M$, where  $M$ is defined in~\eqref{eq:gradBound}.} In this way, we can conclude that there exists finite $k_1$ such that for all $t\ge k_1$,
	\begin{align*}\ex{f(\bxtav)}-f^*
		&\ge -M\, \an{\ex{\dist(\bxtav,X)}}\cr
		&\ge -M\Big(\frac{\bar{S}_{k_1}}{\bar{S}_t}\ex{\dist(\bar{x}_{k_1}^{\textrm av},X)} 
		+\frac{m\b\g_{k_1}^{-1}s_{k_1-1}^{-2}}{\bar{S}_t}\ex{\dist^2(x_{k_1},X)}\cr
		&\quad+\frac{m\b\sum_{k=k_1}^t \g_k^{-1}c_k}{\bar{S}_t}\Big).
	\end{align*} 
	By using the estimation analogous to one in~\eqref{eq:barS_t} and~\eqref{eq:c_t}, we conclude existence of some constant $D_2>0$ such that
	\begin{align}\label{eq:lower1}
		\ex{f(\bxtav)}-f^*\ge -\frac{D_2\ln^g t}{t}.
	\end{align} 
	The constant $D_2$ in its turn depends on the index $k_1$, \an{the upper bound $M$ of the subgradient norms}, number of the constraints $m$, $\a_{\min}$, the Hoffman constant $\beta$, and distance between $x_{k_1}$ and the feasible set $X$.
	Hence, combining~\eqref{eq:upper1} and~\eqref{eq:lower1}, we conclude the result.
\end{proof}
The constant in the $O$-notation can be obtained using Lemma~\ref{lem-prel-est-strong}. In particular, it can be seen that, it depends \emph{linearly} on the  strong-convexity constant $\mu$ and its inverse counterpart $\mu^{-1}$, the squared upper bound of the gradients over a region containing the iterates and their projections on the constraint set $X$ (see~\eqref{eq:gradBound}),  the 
	Hoffman constant $\beta$, the number of the constraints $m$, 
	the inverse of the smallest norm $\a^{-1}_{\min}$ of the vectors $a_i,i\in[m]$, and the expected distance $\ex{\|x_1-x^*\|^2}$ between the initial point $x_1$ and the optimal solution $x^*$. Given the choice of the parameters, we conclude that the rate achieved by the proposed procedure in the case of strongly convex $f$ is $O\left(\frac{\ln^{\e} t}{t}\right)$ for any small positive $\e$.
\begin{rem}
	\an{As seen in Propositions~\ref{prop-converg_rate} and~\ref{prop-converg_rate-strong}, we obtain the convergence rate that is within a logarithmic factor matching the best convergence rates of $O(1/\sqrt{t})$ and $O(1/t)$ known for a stochastic subgradient method applied to a convex and strongly convex function respectively. The extra logarithmic factor is incurred due to the fact that the method \eqref{eq:gradmet} uses stochastic subgradients of time-varying functions $F_k(\cdot)$ yielding the sequence of infeasible iterates.
		As seen from 	the proofs of these results, the extra 
		logarithmic factor comes from the penalty parameters,
		which control the infeasibility of the iterates. However, it is worth noting that the factors $\ln^{1/2+\e} t$ and $\ln^{\e} t$ for any small positive $\e$ in the merely convex and strongly convex cases respectively improve the previously obtained results on convergence rates of incremental procedures with many constraints \cite{Fercoq2019AlmostSC}.  
	}
\end{rem}

\section{Simulation Results}

\begin{figure}[t]
	\centering
	\begin{overpic}[width=1\textwidth]{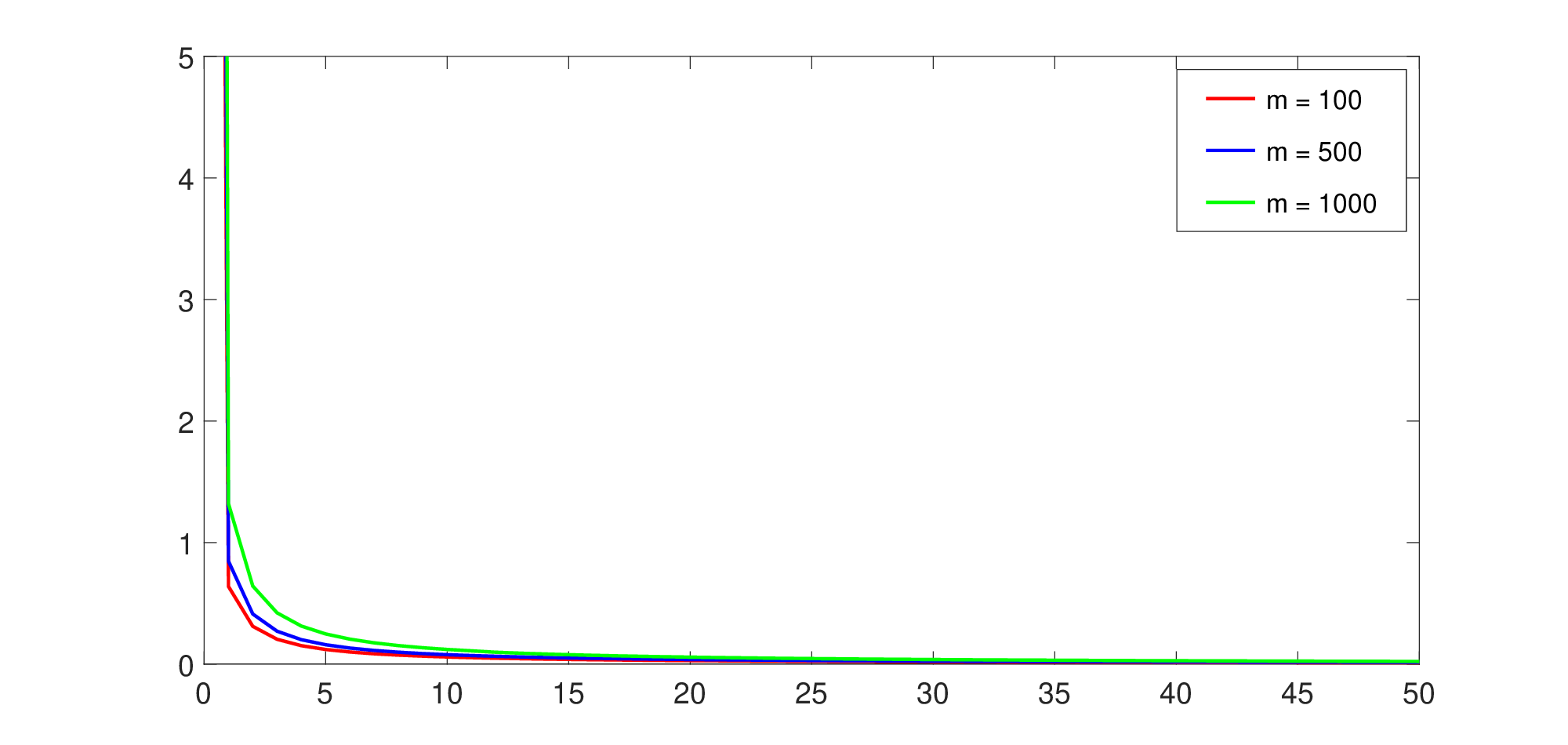}
		\put(6,18){\rotatebox{90}{{Relative Error}}}
		\put(40,-1){{time, $k$}}
	\end{overpic}
	\caption{Strongly convex function $f$ with the optimum inside the feasible set $X$.}
	\label{fig:strInside}
\end{figure}

\begin{figure}[t]
	\centering
	\begin{overpic}[width=1\textwidth]{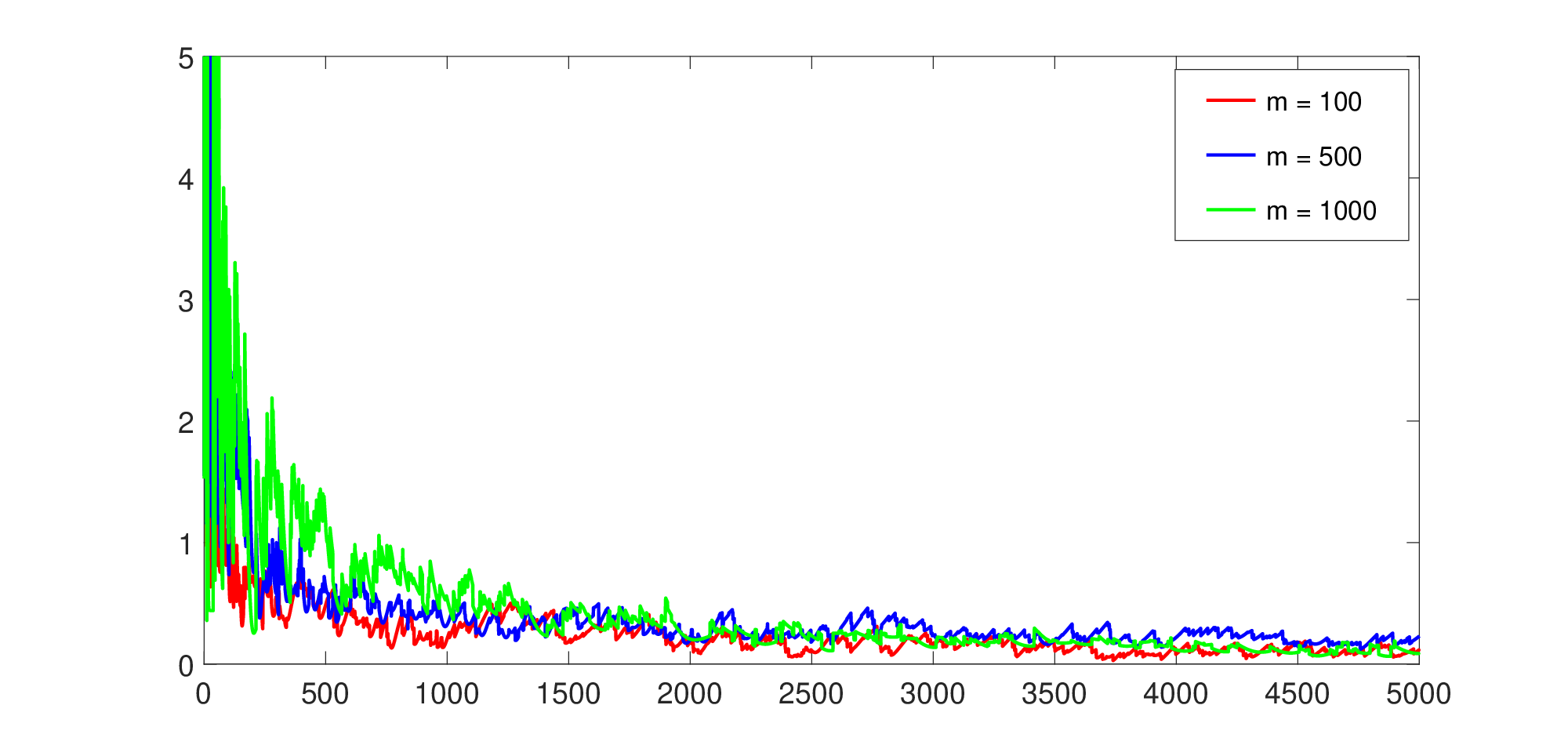}
		\put(6,18){\rotatebox{90}{{Relative Error}}}
		\put(40,-1){{time, $k$}}
	\end{overpic}
	\caption{Strongly convex function $f$ with the optimum outside the feasible set $X$.}
	\label{fig:strOut}
\end{figure}

\begin{figure}[t]
	\centering
	\begin{overpic}[width=1\textwidth]{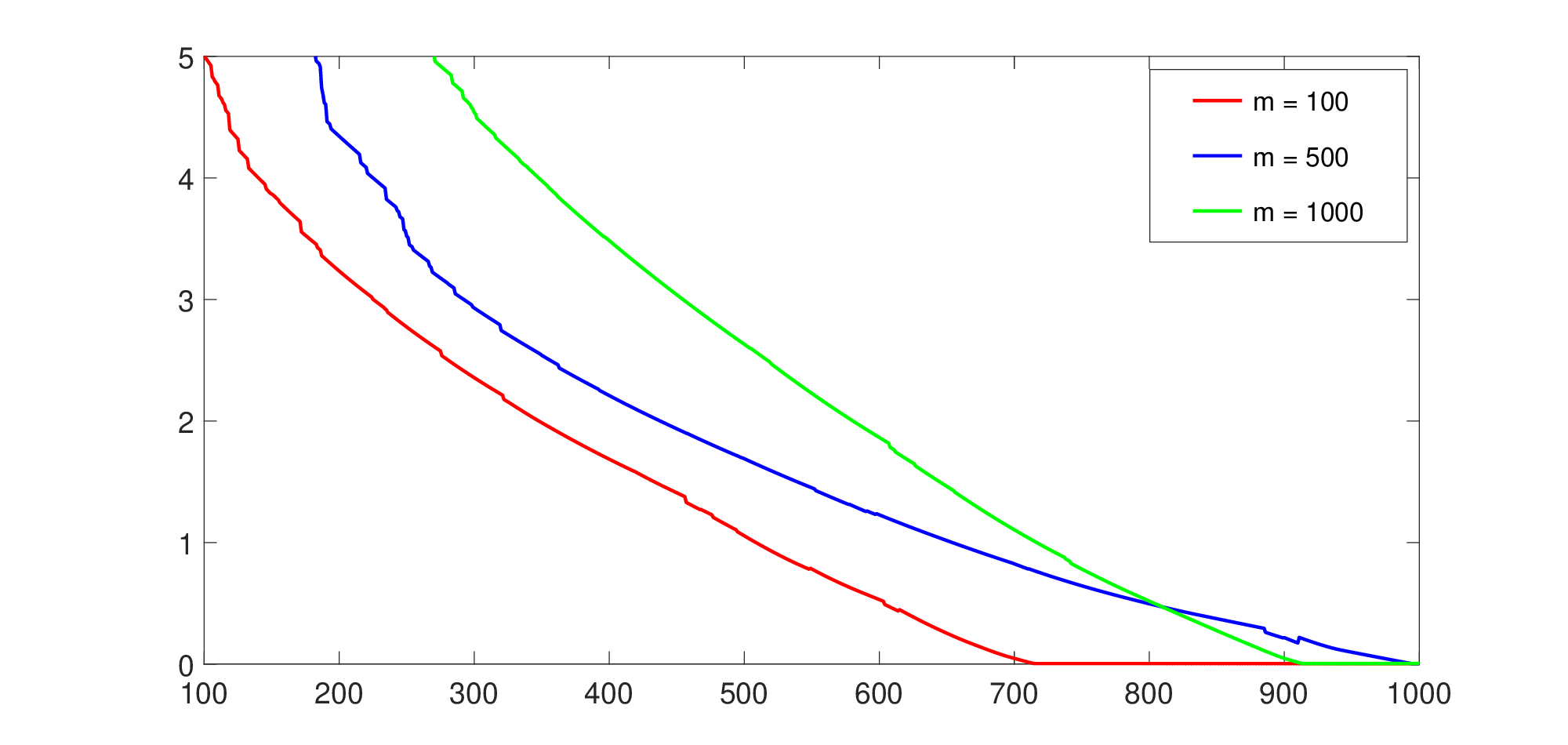}
		\put(6,18){\rotatebox{90}{{Relative Error}}}
		\put(40,-1){{time, $k$}}
	\end{overpic}
	\caption{Merely convex function $f$ with the optimum inside the feasible set $X$.}
	\label{fig:convInside}
\end{figure}

\begin{figure}[t]
	\begin{overpic}[width=1\textwidth]{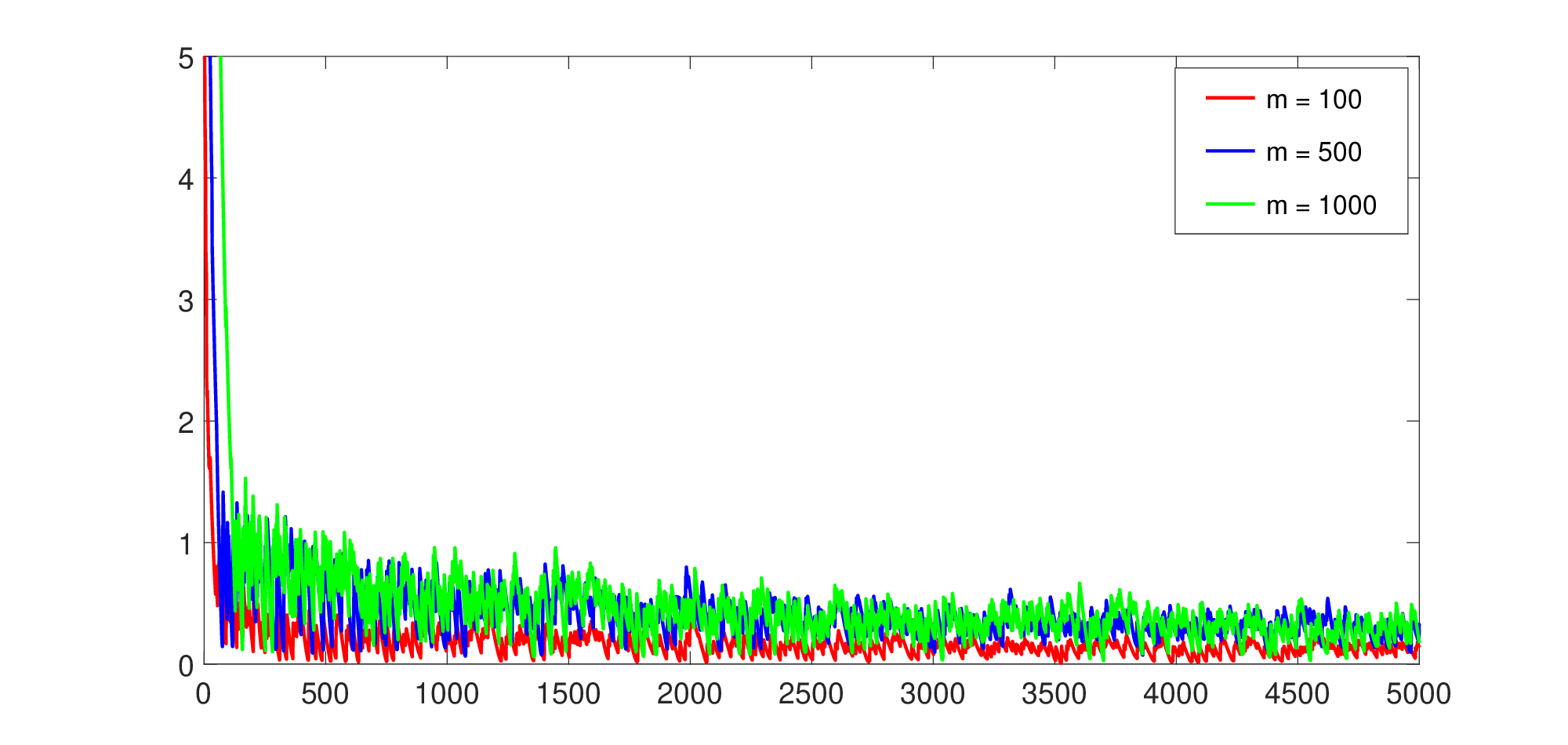}
		\put(6,18){\rotatebox{90}{{Relative Error}}}
		\put(40,-1){{time, $k$}}
	\end{overpic}
	\caption{Merely convex function $f$ with the optimum outside the feasible set $X$.}
	\label{fig:convOut}
\end{figure}

\begin{figure}[t]
	\begin{overpic}[width=1\textwidth]{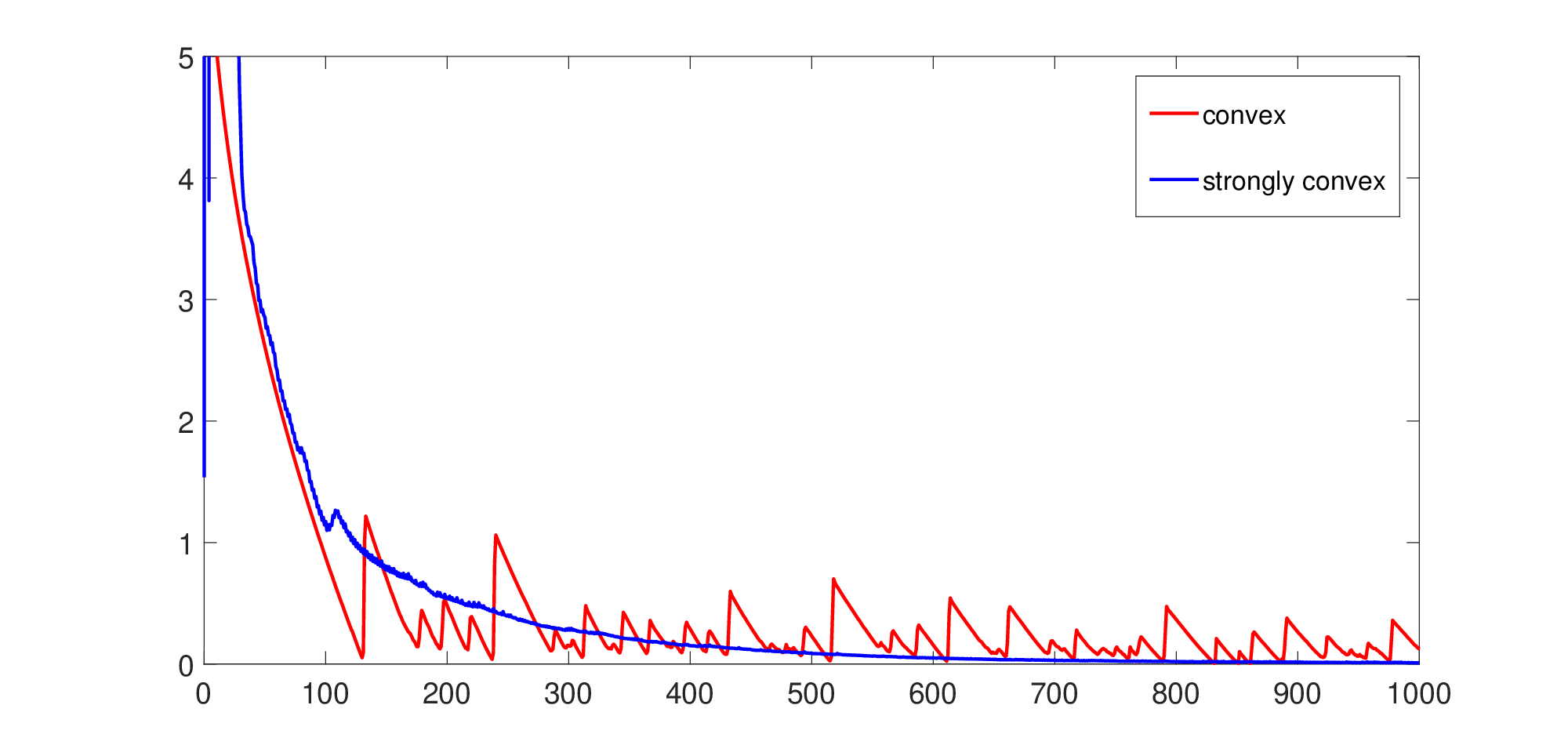}
		\put(6,18){\rotatebox{90}{{Relative Error}}}
		\put(40,-1){{time, $k$}}
	\end{overpic}
	\caption{Full gradient method for the case of optimum outside the feasible set $X$.}
	\label{fig:Fullconv}
\end{figure}
In this section we present some simulation results obtained for the proposed procedure~\eqref{eq:gradmet}.
To implement the algorithm for a strongly convex problem, we set up the objective function as follows:
\[f(x) = \|x-x^0\|^2,\]
where $f:\R^{10}\to\R$ and $x^0\in\R^{10}$ is some randomly generated vector. The vectors $a_i\in\R^{10}$ and the scalars $b_i$, $i\in\{1,\ldots,m\}$, defining the constraint set $X$, are generated according to a normal distribution. We distinguish between two settings: 1) $x^0\in\R^{10}$, as well as $a_i\in\R^{10}$ and $b_i\in\R$, are generated in such a way that $x^0\in X$; 2) $x^0\in\R^{10}$, as well as $a_i\in\R^{10}$ and $b_i\in\R$, are generated in such a way that $x^0\notin X$. The simulation results in these cases, given $m=100, 500, 1000$, are demonstrated by Figures~\ref{fig:strInside} and~\ref{fig:strOut} respectively. In these simulations, we estimate the relative error in terms of the actual iterates, namely $\frac{\|x_k - x^*\|}{\|x^*\|}$. The parameters $s_t$, $\g_t$, and $\d_t$ are chosen according to the conditions in Proposition~\ref{prop-converg_rate-strong}, namely $s_t = \frac{1}{k^{0.99}}$, $\g_t = c\ln k$, $s_t = \frac{1}{k^{2}}$, where the parameter $c$ is tuned. 

For the case of non-strongly convex optimization, we choose the following objective function: 
\[f(x) = \|x-x^0\|_{1},\]
where $\|\cdot\|_1$ denotes the $l_1$-norm. Analogously to the strongly convex case, $f:\R^{10}\to\R$ and $x^0\in\R^{10}$ is some randomly generated vector. The vectors $a_i\in\R^{10}$ and the scalars $b_i$, $i\in\{1,\ldots,m\}$, defining the constraint set $X$, are generated according to a normal distribution. As before, we distinguish between two settings: 1) $x^0\in\R^{10}$, as well as $a_i\in\R^{10}$ and $b_i\in\R$, are generated in such a way that $x^0\in X$; 2) $x^0\in\R^{10}$, as well as $a_i\in\R^{10}$ and $b_i\in\R$, are generated in such a way that $x^0\notin X$. The simulation results in these cases, given $m=100, 500, 1000$, are demonstrated by Figures~\ref{fig:convInside} and~\ref{fig:convOut} respectively. Here, the parameters $s_t$, $\g_t$, and $\d_t$ are chosen according to the conditions in Proposition~\ref{prop-converg_rate}, namely $s_t = \frac{S}{k^{0.5}}$, $\g_t = c\ln k$, $s_t = \frac{1}{k^{2}}$, where the parameters $S$ and $c$ are tuned. 

For the completeness of simulations, we implemented the non-incremental version of the procedure~\eqref{eq:gradmet}, where at each iteration the full gradient of the penalized objective function is used. The corresponding results are presented by Figure~\ref{fig:Fullconv}.


\section{Conclusion}\label{sec-concl}
\tatiana{This work deals with penalty reformulation of the optimization problems subjected to linear constraints, where the penalty is a variant of the one-sided Huber loss function. The infeasibility properties of the solutions of penalized problems for nonconvex and convex objective functions under time-varying penalty parameters are analyzed. 
	A random incremental penalty method for solving convex problem is proposed. This method is proven to converge to a solution of the original problem almost surely and in expectation for suitable choices of the penalty parameters and the stepsize. Moreover,  $O(\ln^{1/2+\e} k/{\sqrt k})$-convergence rate when the objective function is convex, and  $O(\ln^{\e} k/k)$-convergence rate when the objective function is strongly convex are established, given any positive $\e$. }

Some interesting questions for the future work include applicability of
accelerated incremental algorithms for the proposed penalty
reformulation in the case of both strongly and non-strongly
convex optimization as well as extension of the presented penalty approach to other structures of constraints.

\bibliographystyle{plain}
\bibliography{Literature}

\begin{thebibliography}{10}

\bibitem{filter}
J.~W. Adams.
\newblock $\mbox{FIR}$ digital filters with least-squares stopbands subject to
  peak-gain constraints.
\newblock {\em IEEE Transactions on Circuits and Systems}, 38(4):376--388,
  1991.

\bibitem{BertsekasConstrOpt}
D.~P. Bertsekas.
\newblock {\em Constrained Optimization and Lagrange Multiplier Methods
  (Optimization and Neural Computation Series)}.
\newblock Athena Scientific, 1 edition, 1996.

\bibitem{Ber97}
D.~P. Bertsekas.
\newblock A hybrid incremental gradient method for least squares.
\newblock {\em SIAM Journal on Optimization}, 7:913--926, 1997.

\bibitem{Bertsekas2011}
D.~P. Bertsekas.
\newblock Incremental proximal methods for large scale convex optimization.
\newblock {\em Mathematical Programming}, 129(2):163--195, 2011.

\bibitem{BertsekasPenalty}
D.~P. Bertsekas.
\newblock Incremental gradient, subgradient, and proximal methods for convex
  optimization: {A} survey.
\newblock available on arxiv at https://arxiv.org/abs/1507.01030, 2015.

\bibitem{Billingsley}
P.~Billingsley.
\newblock {\em Probability and Measure}.
\newblock John Wiley \& Sons, Inc., New York, NY, USA, 3rd edition, 1995.

\bibitem{SmartG}
G.~Dorini, P.~Pinson, and H.~Madsen.
\newblock Chance-constrained optimization of demand response to price signals.
\newblock {\em IEEE Transactions on Smart Grid}, 4(4):2072--2080, 2013.

\bibitem{Fercoq2019AlmostSC}
O.~Fercoq, A.~Alacaoglu, I.~Necoara, and V.~Cevher.
\newblock Almost surely constrained convex optimization.
\newblock In {\em ICML}, 2019.

\bibitem{GGM06}
M.~Gaudioso, G.~Giallombardo, and G.~Miglionico.
\newblock An incremental method for solving convex finite min-max problems.
\newblock {\em Mathematics of Operations Research}, 31:173--187, 2006.

\bibitem{Gri94}
L.~Grippo.
\newblock A class of unconstrained minimization methods for neural network
  training.
\newblock {\em Optimization Methods and Software}, 4:135--150, 1994.

\bibitem{Gri00}
L.~Grippo.
\newblock Convergent on-line algorithms for supervised learning in neural
  networks.
\newblock {\em IEEE Transactions on Neural Networks}, 11:1284--1299, 2000.

\bibitem{HeD09}
E.~S. Helou and A.~R. De~Pierro.
\newblock Incremental subgradients for constrained convex optimization, a
  unified framework and new methods.
\newblock {\em SIAM Journal on Optimization}, 20:1547--1572, 2009.

\bibitem{Hoffman1952}
A.~J.\ Hoffman.
\newblock On approximate solutions of systems of linear inequalities.
\newblock {\em Journal of Research of the National Bureau of Standards},
  49(4):263--265, 1952.

\bibitem{JRJ09}
B.~Johansson, M.~Rabi, and M.~Johansson.
\newblock A randomized incremental subgradient method for distributed
  optimization in networked systems.
\newblock {\em SIAM Journal on Optimization}, 20:1157--1170, 2009.

\bibitem{Kibardin}
V.~M. Kibardin.
\newblock Decomposition into functions in the minimization problem.
\newblock {\em Automation and Remote Control}, 40:1311--1323, 1980.

\bibitem{Kundu2018}
A.\ Kundu, F.\ Bach, and C.\ Bhattacharyya.
\newblock Convex optimization over intersection of simple sets: improved
  convergence rate guarantees via an exact penalty approach.
\newblock In {\em Proceedings of the 21st International Conference on
  Artificial Intelligence and Statistics}, volume~84 of {\em Proceedings of
  Machine Learning Research}, pages 958--967. PMLR, 2018.

\bibitem{Huber}
W.~Li and J.~Swetits.
\newblock The linear $\ell$-1 estimator and the {H}uber {M}-estimator.
\newblock {\em SIAM Journal on Optimization}, 8(2):457--475, 1998.

\bibitem{Luo91}
Z.~Q. Luo.
\newblock On the convergence of the lms algorithm with adaptive learning rate
  for linear feedforward networks.
\newblock {\em Neural Computation}, 3:226--245, 1991.

\bibitem{clustering}
C.~Mathieu and W.~Schudy.
\newblock Correlation clustering with noisy input.
\newblock In {\em {SODA}}, pages 712--728. {SIAM}, 2010.

\bibitem{NecRichPat2019}
I.~Necoara, P.~Richt\'arik, and A.~Patrascu.
\newblock Randomized projection methods for convex feasibility: {C}onditioning
  and convergence rates.
\newblock {\em SIAM Journal on Optimization}, 29(4):2814--2852, 2019.

\bibitem{Nedic-cdc-2010}
A.~Nedi\'c.
\newblock Random projection algorithms for convex set intersection problems.
\newblock In {\em Proceedings of the 49th IEEE Conference on Decision and
  Control}, pages 7655--7660, 2010.

\bibitem{Nedich2011}
A.~Nedi{\'{c}}.
\newblock Random algorithms for convex minimization problems.
\newblock {\em Mathematical Programming}, 129(2):225--253, 2011.

\bibitem{NeB01}
A.~Nedi\'c and D.~P. Bertsekas.
\newblock Incremental subgradient methods for nondifferentiable optimization.
\newblock {\em SIAM Journal on Optimization}, 12:109--138, 2001.

\bibitem{ned2014}
A.~Nedi\'c and S.~Lee.
\newblock On stochastic subgradient mirror-descent algorithm with weighted
  averaging.
\newblock {\em SIAM Journal on Optimization}, 24(1):84--107, 2014.

\bibitem{PenaltyConvRate-cdc}
A.\ Nedi\'c and T.\ Tatarenko.
\newblock Convergence rate of a penalty method for strongly convex problems
  with linear constraints.
\newblock In {\em Proceedings of the 59th IEEE Conference on Decision and
  Control, Jeju Island, Republic of Korea, Dec.\ 14--18, 2020}, pages 372--377,
  2020.

\bibitem{Patrascu2017NonAsConv}
A.~Patrascu and I.~Necoara.
\newblock Nonasymptotic convergence of stochastic proximal point methods for
  constrained convex optimization.
\newblock {\em J. Mach. Learn. Res.}, 18(1):7204–7245, jan 2017.

\bibitem{pena2021}
J.~Pena, J.~Vera, and L.~F. Zuluaga.
\newblock New characterizations of hoffman constants for systems of linear
  constraints.
\newblock {\em Mathematical Programming}, 187:79--109, 2021.

\bibitem{polyak87}
B.~T. Polyak.
\newblock {\em Introduction to optimization}.
\newblock Optimization Software, Inc., New York, 1987.

\bibitem{Robbins1971}
H.~Robbins and D.~Siegmund.
\newblock A convergence theorem for nonnegative almost supermartingales and
  some applications.
\newblock In {\em Optimizing Methods in Statistics}, pages 233--257. Academic
  Press, New York, 1971.

\bibitem{Siedlecki}
W.~Siedlecki and J.~Sklansky.
\newblock Constrained genetic optimization via dynamic reward-penalty balancing
  and its use in pattern recognition.
\newblock In {\em Proceedings of the Third International Conference on Genetic
  Algorithms}, pages 141--150, San Francisco, CA, USA, 1989. Morgan Kaufmann
  Publishers Inc.

\bibitem{Sol98}
M.~V. Solodov.
\newblock Incremental gradient algorithms with stepsizes bounded away from
  zero.
\newblock {\em Comput. Opt. Appl.}, 11:28--35, 1998.

\bibitem{Song2021}
C.~Song, Cheuk~Yin Lin, Stephen~J. Wright, and Jelena Diakonikolas.
\newblock Coordinate linear variance reduction for generalized linear
  programming, 2021.

\bibitem{Strohmer}
T.~Strohmer and R.~Vershynin.
\newblock A randomized kaczmarz algorithm with exponential convergence.
\newblock {\em Journal of Fourier Analysis and Applications}, 15(2), 2008.
\newblock article number 262.

\bibitem{Penalty_siam}
T.\ Tatarenko and A.\ Nedi\'c.
\newblock A smooth inexact penalty reformulation of convex problems with linear
  constraints.
\newblock {\em SIAM Journal on Optimization}, 31(3):2141--2170, 2021.

\bibitem{Tse98}
P.~Tseng.
\newblock An incremental gradient(-projection) method with momentum term and
  adaptive stepsize rule.
\newblock {\em SIAM Journal on Optimization}, 8:506--531, 1998.

\bibitem{WangBertsSM}
M.~Wang and D.~P. Bertsekas.
\newblock Incremental constraint projection methods for variational
  inequalities.
\newblock {\em Mathematical Programming}, 150(2):321–363, 2015.

\bibitem{Wright08}
S.~J. Wright, R.~D. Nowak, and M.~A.~T. Figueiredo.
\newblock Sparse reconstruction by separable approximation.
\newblock In {\em Proceedings of the IEEE International Conference on
  Acoustics, Speech and Signal Processing (ICASSP)}, pages 3373--3376, 2008.

\bibitem{Xu2020PrimalDualSG}
Y.~Xu.
\newblock Primal-dual stochastic gradient method for convex programs with many
  functional constraints.
\newblock {\em SIAM J. Optim.}, 30:1664--1692, 2020.

\bibitem{matching}
M.~Zaslavskiy, F.~Bach, and J.~P. Vert.
\newblock A path following algorithm for the graph matching problem.
\newblock {\em IEEE Transactions on Pattern Analysis and Machine Intelligence},
  31(12):2227--2242, 2009.

\end{thebibliography}


\newpage
\appendix

\section{Proof of Proposition~\ref{prop-distancetofeas}}\label{app:distancetofeas}
\begin{proof}
	Let $\hat x\in X$ be an arbitrary feasible point, and $k\ge1$ be arbitrary but fixed. 
	By the optimality of $x_k^*$, we have 
	\[F_k(x_k^*)\le F_k(\hat x).\]
	By using $F_k(x)=f(x)+\g_k H_k(x)$ (see~\eqref{eq-penfun-hk}),
	after re-arranging the terms, we obtain
	\[\g_kH_k(x_k^*)\le f(\hat x) - f(x_k^*)+\g_k H_k(\hat x).\]
	By relation~\eqref{eq-hk-atfeasx} with $x=\hat x\in X$ and 
	relation~\eqref{eq-hk-atanyx} with $x=x^*_k$, we obtain 
	\[\frac{\g_k}{m\b}\dist(x_k^*,X) \le f(\hat x) - f(x_k^*)+\frac{\g_k\d_k}{4\a_{\min}}.\]
	By multiplying the preceding relation with $m\b/\g_k$, we have
	\begin{equation}\label{eq-disto}
		\dist(x_k^*,X) \le \frac{m\b}{\g_k} (f(
		\hat x) - f(x_k^*))+\frac{m\b\d_k}{4\a_{\min}},\end{equation}
	which is the stated relation. 
	By Lemma~\ref{lem-solsbded}, the optimal sets $X_k^*$ are uniformly bounded, i.e., there exist some $D>0$, such that 
	$\|x_k^*\|\le D$ for all $x_k^*\in X_k^*$ and for all $k\ge1$. Hence, the values $f(x_k^*)$ are also uniformly bounded
	implying that 
	$\frac{m\b}{\g_k} (f(\hat x) - f(x_k^*))\to0$ and $\g_k\to\infty$. The rest of the statement follows immediately from relation~\eqref{eq-disto}.
\end{proof}

\section{Proof of Lemma~\ref{lem-solsbded}}\label{app:solsbded}
\begin{proof}
	The set $X_k^*$ is closed by the continuity of $F_k(\cdot)$ for all~$k$. It remains to show that $X_k^*$ is bounded and nonempty for all $k$. Let $k\ge 1$ be arbitrary. 
	By Corollary~\ref{cor:lset}, with $\d=\d_k$, 
	$\g=\g_k$, and $\hat x\in X$, we have
	\begin{align}\label{eq-bounded}
		\{x\in\R^n\mid F_k(x)\le t_{k}(\hat x)\}\subseteq
		\{x\in\R^n\mid f(x)\le t_{k}(\hat x)\},\end{align}
	with $t_k(\hat x)= f(\hat x)+\g_k\d_k/(4\a_{\min}).$
	Since $\g_k\d_k\le c$, it follows that 
	\[t_{k}(\hat x)\le t_c(\hat x)\triangleq 
	f(\hat x)+\frac{c}{4\a_{\min}},\]
	implying that 
	\[\{x\in\R^n\mid F_k(x)\le t_{k}(\hat x)\}\subseteq
	\{x\in\R^n\mid f(x)\le  t_c(\hat x)\}.\]
	The set $\{x\in\R^n\mid f(x)\le t_{c}(\hat x)\}$
	is bounded since $f$ has bounded lower-level sets, implying by the preceding relation that the set $\{x\in\R^n\mid F_k(x)\le t_{k}(\hat x)\}$ is bounded. Since
	$X_k^*\subseteq \{x\in\R^n\mid F_k(x)\le t_{k}(\hat x)\}$, the set $X_k^*$ is bounded.
	
	To show that $X_k^*\ne\emptyset$, we observe that the penalized problem 
	$\min_{x\in \R^n} F_k(x)$ and its related constrained problem
	\[\min_{F_k(x)\le t_{k}(\hat x)} F_k(x)\]
	have the same optimal values and the same optimal sets.
	The optimal set of the problem above is nonempty by the Weierstrass Theorem, implying that $X_k^*$ is nonempty.
	
	To show that $\{X_k^*\}$ is uniformly bounded, we use the fact that 
	$X_k^*\subseteq\{x\in\R^n\mid F_k(x)\le t_{k}(\hat x)\}$ for all $k$, which by~\eqref{eq-bounded} implies that 
	$X_k^*\subseteq\{x\in\R^n\mid f(x)\le t_{c}(\hat x)\}$ for all $k$.
	The boundedness of the set $\{x\in\R^n\mid f(x)\le t_{c}(\hat x)\}$ implies that $\{X_k^*\}$ is uniformly bounded. 
\end{proof}

\section{Proof of Proposition~\ref{prop-sols}}\label{app:sols}
\begin{proof}
	In case $\mu=0$, the conditions of Lemma~\ref{lem-solsbded} are satisfied.
	If $f(\cdot)$ is strongly convex with $\mu>0$, then it has bounded lower-level sets.
	In this case, the conditions of Lemma~\ref{lem-solsbded} are also satisfied.
	Thus, by Lemma~\ref{lem-solsbded} the optimal solutions $x_k^*$ of the penalized problems are uniformly bounded, i.e., there exists $D>0$ such that 
	$\|x_k^*\|\le D$ for any optimal solution $x_k^*\in X_k^*$ of the penalized problem $\min_{x\in\R^n} F_k(x)$ 
	and for all $k$. Since 
	$\{x_k^*\mid x_k^*\in X_k^*, k\ge 1\}$
	is bounded, the set $\{\Pi_X[x_k^*]\mid x_k\in X_k^*,\ k\ge1\}$ is also bounded, i.e.,
	there is $R>0$ such that $\Pi_X[x_k^*]\le R$ for all $x_k^*\in X_k^*$ and all $k$.
	Thus, the constant $L$ in~\eqref{eq:sgdbound} is finite.
	
	Since each $F_k(\cdot)$ is strongly convex with a constant $\mu\ge0$,
	by the optimality of $x_k^*$ we have 
	\[\frac{\mu}{2}\|x^* -x_k^*\|^2\le F_k(x_k^*)- F_k(x)\quad\hbox{for any $x\in\R^n$}.\]
	Using $x=x^*$, where $x^*\in X^*$  and the expression for $F_k$ in~\eqref{eq-penfun-hk}, we obtain for all $k$,
	\begin{align}\label{eq-11}
		\frac{\mu}{2}&\|x^* -x_k^*\|^2\le f(x^*)-f(x_k^*) + \g_kH_k(x^*) - \g_kH_k(x_k^*).
	\end{align}
	Adding and subtracting $f(\Pi_X[x_k^*])$ yields
	\begin{align}\label{eq-12}
		f(x^*)-f(x_k^*)&=f(x^*)-f(\Pi_X[x_k^*]) + f(\Pi_X[x_k^*]) - f(x_k^*)\cr
		&\le -\frac{\mu}{2}\|x^*-  \Pi_X[x_k^*]\|^2
		+f(\Pi_X[x_k^*])-f(x_k^*),
	\end{align}
	where the inequality is obtained using 
	$\frac{\mu}{2}\|\Pi_X[x_k^*]-x^*\|^2 + \la \tilde \nabla f(x^*), \Pi_X[x_k^*]-x^*\ra +f(x^*)\le f(p_k)$ 
	for a subgradient $\tilde \nabla f(x^*)\in\partial f(x^*)$ (see~\eqref{eq-strconvex}),
	and the fact that $\la \tilde \nabla f(x^*), \Pi_X[x_k^*]-x^*\ra\ge0$, 
	which holds since $\Pi_X[x_k^*]$ is feasible
	and $x^*$ is the optimal point of $f(\cdot)$ over $X$. 
	By the (strong) convexity relation~\eqref{eq-strconvex}, we also have
	\begin{align*}
		f(\Pi_X[x_k^*])-f(x_k^*)&\le -\frac{\mu}{2}\|\Pi_X[x_k^*]-x_k^*\|^2 \cr
		&\ + \la \tilde\nabla f(\Pi_X[x_k^*]), \Pi_X[x_k^*] - x_k^*\ra.\end{align*}
	Since the subgradients of $f(\cdot)$ at points $\Pi_X[x^*_k]$ are uniformly bounded (see~\eqref{eq:sgdbound}), 
	it follows that 
	\[f(\Pi_X[x_k^*])-f(x_k^*)\le -\frac{\mu}{2}\|\Pi_X[x_k^*]-x_k^*\|^2 + L\|\Pi_X[x_k^*] - x_k^*\|.\]
	Combining the preceding inequality with~\eqref{eq-11} and~\eqref{eq-12}, after 
	re-arranging nonpositive terms, we obtain
	\begin{align*}
		\frac{\mu}{2}&\|x^* -x_k^*\|^2+\frac{\mu}{2}\|x^*-\Pi_X[x_k^*]\|^2 +\frac{\mu}{2}\|\Pi_X[x_k^*] - x_k^*\|^2\cr
		&\le L \|\Pi_X[x_k^*] - x_k^*\|+ \g_k H_k(x^*)-\g_k H_k(x_k^*).
	\end{align*}
	Since $x^*$ is feasible, we apply relation~\eqref{eq-hk-atfeasx} with $x=x^*$.
	Also, we use~\eqref{eq-hk-atanyx} with $x=x_k^*$, and obtain
	\begin{align*}
		\frac{\mu}{2}&\|x^* -x_k^*\|^2+\frac{\mu}{2}\|x^*-\Pi_X[x_k^*]\|^2 +\frac{\mu}{2}\|\Pi_X[x_k^*] - x_k^*\|^2\cr
		&\le L \|\Pi_X[x_k^*] - x_k^*\| + \frac{\g_k\d_k}{4\a_{\min}} -\frac{\g_k}{m\b} \dist(x_k^*,X).
	\end{align*}
	Using
	$\|\Pi_X[x_k^*] -x_k^*\|=\dist(x_k^*,X)$ and grouping the terms accordingly, we arrive at the stated relation.
\end{proof}



\end{document}